\documentclass[a4paper,12pt]{article}

\usepackage{mathtools, bbm, bm,array,amsmath,amsthm}
\usepackage{tikz}
\usepackage{caption}
\usepackage{environ}
\usepackage{graphicx}
\usetikzlibrary{shapes}
\usetikzlibrary{decorations.markings}
\usetikzlibrary{patterns}

\makeatletter
\newcommand\footnoteref[1]{\protected@xdef\@thefnmark{\ref{#1}}\@footnotemark}
\makeatother

\newtheorem{theorem}{Theorem}%
\newtheorem{lemma}[theorem]{Lemma}%
\newtheorem{corollary}[theorem]{Corollary}%

\let\originalleft\left
\let\originalright\right
\renewcommand{\left}{\mathopen{}\mathclose\bgroup\originalleft}
\renewcommand{\right}{\aftergroup\egroup\originalright}

\newcommand{\N}{\mathbb{N}}

\newcommand{\R}{\mathbb{R}}

\newcommand{\diff}{\mathop{}\!d}
\newcommand{\ecfup}{M^{\uparrow}}
\newcommand{\volup}{V^{\uparrow}}

\newcommand{\symbolnotsobw}{\symbolnot_{\mathcal{W}}}
\newcommand{\symbolnotsch}{\symbolnot_{\mathcal{S}}}
\newcommand{\ballsch}{\mathcal{B}^{(s)}_{\mathcal{S}}}
\newcommand{\ballweight}{\mathcal{B}^{(s,r)}_{\mathcal{W}}}
\newcommand{\tsch}{T_{\mathcal{S}}}

\newcommand{\psidosymbclass}{\Gamma^m_\rho}
\newcommand{\hyposymbclass}{\text{H}\Gamma^{m_-,m_+}_\rho}
\newcommand{\hyposymbclassinv}{\mathrm{H}\Gamma^{-m_+,-m_-}_\rho}
\newcommand{\psidoopclass}{\Psi^m_\rho}
\newcommand{\hyponegclass}{\mathrm{H}\Gamma^-}
\newcommand{\hypoposclass}{\mathrm{H}\Gamma^+}
\newcommand{\symbolnot}{\sigma}

\DeclareMathAlphabet{\pazocal}{OMS}{zplm}{m}{n}

\usepackage{pdfsync}

\usepackage{amsmath}
\usepackage{amsfonts}
\usepackage{amssymb}
\usepackage{amsthm}
\usepackage{mathrsfs}
\newtheorem*{lemma*}{Lemma}
\newtheorem*{proposition*}{Proposition}
\newtheorem*{theorem*}{Theorem}
\usepackage[normalem]{ulem}
\usepackage{extarrows}
\usepackage{xcolor}
\usepackage[margin=1cm]{caption}
\usepackage{subcaption}
\usepackage[inline]{enumitem}

\newlength{\mylabelwidth}
\newenvironment{citemize}
  {\begin{list}{}%
     {\setlength{\labelwidth}{\mylabelwidth}%
      \setlength{\leftmargin}{\dimexpr\labelwidth+\labelsep\relax}%
      }}
  {\end{list}}

\usepackage[numbers, sort, compress]{natbib}
\DeclareUnicodeCharacter{025B}{\ensuremath{\varepsilon}}

\usepackage{framed}

\usepackage[linkcolor=blue,colorlinks=true,citecolor=blue,filecolor=black]{hyperref}
\usepackage{geometry}
\usepackage{setspace}

\title{Entropy and Minimax Risk of Hypoelliptic Pseudodifferential Operators}
\date{}
\author{Thomas Allard \\ tallard@ethz.ch  \and Helmut Bölcskei \\ hboelcskei@ethz.ch}

\begin{document}

\maketitle

\abstract{
\noindent
We characterize the 
entropy and
minimax risk of a broad class of compact pseudodifferential operators.
Under suitable decay and regularity conditions on the symbol, 
we combine a Weyl-type asymptotic relation between the 
eigenvalue-counting function and the phase-space volume of the symbol 
with a general correspondence between spectral
quantities, entropy, and minimax risk for compact operators.
This approach yields explicit
asymptotic formulae for both entropy and minimax risk directly in terms
of the symbol.
As an application, we derive sharp entropy and minimax risk asymptotics for
unit balls in Sobolev spaces on unbounded domains, thereby extending
Pinsker’s theorem for Sobolev classes beyond the bounded-domain setting,
and showing that the sharp asymptotic constants are determined by
phase-space geometry rather than domain geometry.
}

\vspace{1cm}

\noindent
{\em H. B\"{o}lcskei dedicates this paper to Prof. Thomas Kailath on the occasion of his 90th birthday.}

\section{Introduction}

Metric entropy provides a natural measure of the massiveness of classes of mathematical objects—such as functions, dynamical systems, or statistical estimators—by quantifying the number of bits required to approximate their elements uniformly to a given precision. 
This concept plays a central role in a wide range of mathematical fields, including approximation theory \cite{lorentzApproximationFunctions1966,lorentzMetricEntropyApproximation1966,lorentzConstructiveApproximationAdvanced1996,cohenTreeApproximationOptimal2001, cohen2022optimal}, 
harmonic analysis \cite{donohoUnconditionalBasesAre1993,donohoUnconditionalBasesBitLevel1996,donohoDataCompressionHarmonic1998,Grohs2015},
high-dimensional statistics and probability \cite{Vershynin_2018,wainwrightHighDimensionalStatistics2019}, 
and machine learning \cite{elbrachterDeepNeuralNetwork2021,hutterMetricEntropyLimits2022,20.500.11850/734099}. The theory of
metric entropy itself has also attracted renewed interest in recent years  \cite{expdecaypaper,polydecaypaper,ab2025compactoperators,allard2025metricentropyminimaxrisk,20.500.11850/734099}; this line of work goes beyond purely scaling-based considerations and develops a more quantitative framework for complexity estimation.

Minimax risk in non-parametric estimation provides a complementary notion of complexity, quantifying the intrinsic difficulty of recovering elements of a class from noisy observations; see, e.g., \cite{tsybakovIntroductionNonparametricEstimation2009,johnstone2019estimation,nussbaum1999minimax}. While metric entropy and minimax risk are closely related, existing characterizations typically treat them through separate analytical frameworks and under restrictive structural assumptions. The present paper develops a unified framework for characterizing the 
entropy and minimax risk of a large class of compact pseudodifferential operators. Explicit asymptotic characterizations are derived in terms of the associated operator symbols. These quantities characterize operator compactness by describing, respectively, the effective size of the operator image of the unit ball and the intrinsic difficulty of recovering elements of this image from noisy measurements \cite{ab2025compactoperators,allard2025metricentropyminimaxrisk}.

We build on the well-known observation that images of pseudodifferential operators are often localized in phase space, with localization properties governed by the associated operator symbols; see, for example, the discussion of uncertainty principles and phase-space concentration in \cite{fefferman_uncertainty_1983}. This observation motivates characterizing the metric entropy and minimax risk of operator images as a natural and effective way to study the complexity of sets of phase-space-localized signals. Such a viewpoint was adopted by the authors in \cite[Section~3.1]{ab2025compactoperators}, where the metric entropy of signal classes subject to time–frequency concentration constraints is analyzed via the ellipsoidal structure of the image of the unit ball under the Landau–Pollak–Slepian operator. This approach yields improvements over the previously best-known entropy characterizations for both the Landau–Pollak–Slepian operator and Sobolev balls. 

The central contribution of the present paper is to expand this localization-based methodology by dispensing with explicit ellipsoidal structure and developing a symbol-based framework that applies to broad classes of compact pseudodifferential operators, thereby enabling a unified analysis of 
entropy and minimax risk under general phase-space localization conditions encoded by symbol decay and regularity.

For compact pseudodifferential operators with sufficiently regular symbols, phase-space localization at the symbol level is reflected in the spectral 
properties 
of the operator; see, e.g., \cite{fefferman_uncertainty_1983} for the underlying microlocal principle.
A spectral formulation is therefore natural.  In the analysis of 
entropy and minimax risk for compact linear operators, it suffices to restrict attention to operators that are positive and self-adjoint. This follows from the fact that any compact operator admits a polar decomposition, and that metric entropy and minimax risk depend only on the geometry of the image of the unit ball, which is preserved under the isometry induced by the polar decomposition 
(see, e.g., \cite{ab2025compactoperators}, \cite{prosserEEntropyECapacityCertain1966},  or \cite[Chapter~3.4]{carlEntropyCompactnessApproximation1990}).
Positive self-adjoint compact operators $T$ admit a spectral decomposition; the corresponding nonzero eigenvalues $\{\lambda_n\}_{n\in\N^*}$, listed in non-increasing order and counted with 
multiplicity, are nonnegative and tend to zero.
The spectral properties of $T$ are typically characterized through the eigenvalue-counting function
\begin{equation}\label{eq:definition-ecf}
    M_T(\lambda) \coloneqq  \# \left\{n\in\N^* \mid \lambda_n \geq \lambda \right\},
    \quad \text{for all } \lambda>0.
\end{equation}
It is well established that the entropy of a positive self-adjoint compact linear operator is governed by its spectral properties;
see, e.g., \cite{ab2025compactoperators,prosserEEntropyECapacityCertain1966,carlInequalitiesEigenvaluesEntropy1980,carl1981entropy,konigEigenvalueDistributionCompact1986,edmundsFunctionSpacesEntropy1996}.
Expressing the operator in its eigenbasis turns these spectral properties into a geometric description of the image of the unit ball under the operator. In particular, this image is an ellipsoid whose semi-axes are given by the eigenvalues $\{\lambda_n\}_{n\in\N^*}$, see \cite{ab2025compactoperators}.
As shown in \cite[Theorems~2, 4, and 5]{allard2025metricentropyminimaxrisk}, this ellipsoidal structure yields 
sharp asymptotic characterizations for both the 
entropy $H_T$ and the minimax risk $R_T$ of $T$ via the type-$\tau$ (for $\tau \geq 1$) integrals \begin{equation}\label{eq:type-tau-integrals}
    I_\tau (\varepsilon) \coloneqq \int_{\varepsilon}^{\infty} \frac{M_T(\lambda)}{\lambda^\tau}\diff \lambda, \quad 
    \varepsilon > 0.
\end{equation}
Specifically, the minimax risk $R_T$ 
is asymptotically determined by the type-$2$ and type-$3$ integrals according to
\begin{equation}\label{eq:integral-ecf-risk}
    R_T (\kappa)
    \sim \kappa^2 \varepsilon_\kappa I_2(\varepsilon_\kappa),
    \quad \text{as } \kappa \rightarrow 0,
\end{equation}
where the critical radius $\varepsilon_\kappa$ is defined as the unique solution of
\begin{equation}\label{eq:def-critical-radius}
    \kappa^2 \left(2I_3(\varepsilon_\kappa) - \frac{I_2(\varepsilon_\kappa)}{\varepsilon_\kappa}\right)=1,
    \quad \text{for all } \kappa > 0.
\end{equation}
Moreover, under a mild regularity condition on the eigenvalue-counting function $M_T$,
the 
entropy $H_T$ is asymptotically equivalent to the type-$1$ integral, that is, 
\begin{equation}\label{eq:integral-ecf-entropy}
    H_T(\varepsilon) 
    \sim  I_1(\varepsilon),
    \quad \text{as } \varepsilon \rightarrow 0.
\end{equation}

The relations \eqref{eq:integral-ecf-risk}--\eqref{eq:integral-ecf-entropy} reduce the characterization of the asymptotic behavior of the 
entropy and minimax risk of a compact operator to that of its asymptotic spectral distribution. 
This reduction has already proved effective in \cite[Section~3.2]{ab2025compactoperators} and \cite[Section~5]{allard2025metricentropyminimaxrisk}, where sharp asymptotic characterizations of metric entropy and minimax risk for unit balls in Sobolev spaces on bounded domains were derived.
In that setting, the operator $T$ is related to the inverse of the Laplacian, and the required asymptotics of the eigenvalue-counting function follow from the Weyl law (see, e.g., \cite[Chapter~9.5]{shubinInvitationPartialDifferential2020}) and its Riesz-means counterpart (see \cite{frank2024riesz}).
The goal of the present paper is to extend this spectral-reduction strategy to operator classes for which the asymptotic behavior of the eigenvalue-counting function can be derived from symbol-level information, notably compact pseudodifferential operators.

An important consequence of our results is a conceptual extension of Pinsker’s
theorem \cite{pinsker1980optimal} to unbounded domains. Pinsker’s original result
and its subsequent extensions are formulated in settings where compactness
arises from working on bounded domains, leading to compact embeddings of Sobolev
balls into $L^2$. We show that after restoring compactness on $\R^d$ through
spatial confinement, the same minimax principle continues to govern statistical
estimation. In this setting, the sharp asymptotic constants are determined by
phase-space geometry rather than by domain geometry, indicating that Pinsker’s
theorem reflects a structural feature of nonparametric estimation rather than an
artifact of bounded domains.

Pseudodifferential operators are linear operators acting on the Schwartz space
\[
\mathcal{S}(\R^d)
=
\left\{ f \in C^\infty(\R^d) :
\sup_{x\in\R^d} \left| x^\alpha \partial^\beta f(x) \right| < \infty
\text{ for all } \alpha, \beta \in \N^d \right\},
\]
and are of the form
\begin{equation}\label{eq:kohn-nirenberg-def}
    T_{\symbolnot} f (x)
    =
    \int_{\R^d}
    \symbolnot(x, \omega)\, \hat f (\omega)\, e^{2i\pi x \cdot \omega}\, d\omega,
\end{equation}
where $\symbolnot \in C^\infty(\R^d\times\R^d)$ is the associated symbol and
$x \cdot \omega$ denotes the Euclidean inner product in $\R^d$.
Under appropriate regularity assumptions on $\symbolnot$, such operators 
can be extended compactly to 
$L^2(\R^d)$ and admit a well-developed spectral theory. Pseudodifferential operators were originally introduced in mathematical physics in connection with quantization; see, e.g.,
\cite{weyl1950theory,wong1998weyl}.
They subsequently became central tools in harmonic analysis (see \cite{steinharmonic} and \cite[Chapter~2]{follandharmonic}), time-frequency analysis (see \cite[Chapter 14]{grochenig_foundations_2001}), 
electrical engineering in the context of linear-time-varying systems \cite{zadeh1950determination,zadeh1950frequency}, 
and wavelet theory (\cite[Chapter 2.10]{meyerWaveletsOperators1993}). 
In pure mathematics, pseudodifferential operators play a fundamental role in partial differential equations and microlocal analysis \cite{Taylor+1981,shubin_pseudodifferential_2001,hormander2007analysis,hormanderAnalysisLinearPartial2009,treves2013psido,hinzmicrolocal}, and they form a key ingredient in the proof of the Atiyah-Singer index theorem (see, e.g., \cite{landweber2005k}).

The mapping \eqref{eq:kohn-nirenberg-def} that associates a linear operator to a
symbol is known as the Kohn--Nirenberg quantization.
In addition to \eqref{eq:kohn-nirenberg-def}, two other classical quantization
schemes are commonly used, namely the Weyl quantization and the right
quantization; these are reviewed in Appendix~\ref{sec:review-lit}.
An important aspect of our analysis is that the results do not depend on the
particular choice of quantization; this is formalized in
Theorem~\ref{thm:main-result}.
Accordingly, we often write $H_\sigma$, $M_\sigma$, and $R_\sigma$ for the 
entropy, eigenvalue-counting function, and minimax risk of $T_\sigma$, omitting
explicit reference to the quantization when it is irrelevant.

Many properties of pseudodifferential operators relevant to spectral asymptotics are naturally expressed at the level of the symbol. In particular, for symbols that are positive and decay at infinity, 
a central quantity in our analysis is the phase-space volume above level 
$\lambda>0$, defined by
\begin{equation}\label{eq:definition-volume-function}
    V_{\symbolnot}(\lambda) 
    \coloneqq \int_{\mathbb{R}^{2d}} \mathbbm{1}_{\{\symbolnot(x, \omega)\, > \, \lambda\}} \diff  x\diff \omega.
\end{equation}
The function $V_{\symbolnot}$ records the phase-space distribution of the symbol.
Under the regularity conditions specified in the main results below, the behavior of $V_{\symbolnot}(\lambda)$ as $\lambda \rightarrow 0$ quantifies the decay of $\sigma$. 
As recalled in Theorem~\ref{thm:Dauge-Robert}, this gives rise to a Weyl-type asymptotic for the spectral distribution of $T_{\symbolnot}$, namely 
\begin{equation}\label{eq:volume-ecf-relation}
    M_{\symbolnot}(\lambda)  \sim V_{\symbolnot}(\lambda), \quad 
    \lambda \rightarrow 0. 
\end{equation}
Relation \eqref{eq:volume-ecf-relation} serves as the link between symbol-level quantities and the general 
entropy and minimax-risk characterizations in 
\eqref{eq:integral-ecf-risk}--\eqref{eq:integral-ecf-entropy}. Combining these results leads to the main conclusions of the paper,
stated in Theorem~\ref{thm:main-result} and Corollaries~\ref{cor:ME-Main}--\ref{cor:MR-Main}. In particular, the 
entropy satisfies
 \begin{equation}\label{eq:main-res-entropy}
    H_{\symbolnot}(\varepsilon) 
    \sim \int_{\R^{2d}} \ln_{+}\left(\frac{\symbolnot(x, \omega)}{\varepsilon}\right) \diff  x \diff \omega, 
    \quad \varepsilon \to 0,
\end{equation}
and the minimax risk has the asymptotic behavior 
\begin{equation}\label{eq:main-res-risk-1}
    R_\symbolnot (\kappa)
    \sim \kappa^2 \int_{\R^{2d}} \left(1-\frac{\varepsilon_\kappa}{\symbolnot(x, \omega)}\right)_+ \diff  x \diff \omega, 
    \quad  \kappa \to 0,
\end{equation}
where the critical radius $\varepsilon_\kappa$ is determined by 
\begin{equation}\label{eq:main-res-risk-2}
    \kappa^2 
    \int_{\R^{2d}} \frac{1}{\symbolnot(x, \omega) \, \varepsilon_\kappa }\left(1-\frac{\varepsilon_\kappa}{\symbolnot(x, \omega)}\right)_+ \diff  x \diff \omega =1, 
    \quad  \kappa >0.
\end{equation}
To put \eqref{eq:volume-ecf-relation} into context, we review the classical
Weyl-type theory for elliptic operators in Appendix~\ref{sec:review-lit}.
In Section~\ref{sec:main-results}, we adapt these ideas to the compact
symbol-decay setting considered here and use them to establish the
entropy–volume relation of Theorem~\ref{thm:main-result}. Section~\ref{sec:main-results}
also contains the principal conceptual contributions of the paper, namely the asymptotic formulae \eqref{eq:main-res-entropy}–\eqref{eq:main-res-risk-2} together with their interpretation. Finally, Section~\ref{sec:appli-sobolev-spaces} applies these results to Sobolev spaces on unbounded domains, complementing the corresponding results for bounded domains obtained in \cite[Section~3.2]{ab2025compactoperators} and establishing Pinsker’s theorem in the unbounded setting.

\paragraph{Notation.}
We write $\N$ for the set of natural numbers including zero,
$\N^*$ for the set of natural numbers excluding zero,
$\R$ for the real numbers, and $\R^*_+$ for the positive real numbers.
For $d \in \N^*$, we denote by $\omega_d$ the volume of the unit ball in $\R^d$,
and by $\|z\|_2$ the Euclidean norm of $z \in \R^d$.

$L^2(\R^d)$ and $C^\infty(\R^d)$ designate, respectively, the Lebesgue space of square-integrable functions and the space of infinitely differentiable functions on $\R^d$.
For $f \in C^\infty(\R^d)$, we write $\nabla f$ for its gradient and $\partial_i f$ for its partial derivative with respect to the $i$-th coordinate, where $i \in \{1,\dots,d\}$. A multi-index $\alpha = (\alpha_1,\dots,\alpha_d) \in \N^d$ defines the partial derivative
$\partial^\alpha = \partial_1^{\alpha_1} \cdots \partial_d^{\alpha_d}$,
with order $|\alpha| = \alpha_1 + \cdots + \alpha_d$.
The Fourier transform of a function $f \in L^2(\R^d)$ is denoted by $\hat f \in L^2(\R^d)$.

When comparing the asymptotic behavior of functions
$f,g \colon \R_+^* \to \R_+^*$ as $x \to 0$, we write
$f(x) = o_{x \to 0}(g(x))$ if $\lim_{x \to 0} f(x)/g(x) = 0$, and
$f(x) = O_{x \to 0}(g(x))$ if there exists a constant $C>0$ such that
$\limsup_{x \to 0} f(x)/g(x) \le C$.
Further, $f(x) \sim g(x)$ as $x \to 0$ if $\lim_{x \to 0} f(x)/g(x) = 1$.
We write $f(x) \lesssim g(x)$ as $x \to 0$ if there exist constants
$C>0$ and $x_0>0$ such that $f(x) \le C\, g(x)$ for all $x \in (0,x_0)$,
and $f(x) \gtrsim g(x)$ as $x \to 0$ if $g(x) \lesssim f(x)$.
Moreover, $f(x) \asymp g(x)$ as $x \to 0$
if there exist constants $0<c\le C<\infty$ and $x_0>0$ so that $
c\,g(x) \le f(x) \le C\,g(x), \, x \in (0,x_0)$.

Finally, $\ln(\cdot)$ denotes the natural logarithm,
$\ln_+(x) = \max\{0,\ln(x)\}$ for $x \ge 0$ (with the convention $\ln_+(0)=0$),
$(x)_+ = \max\{0,x\}$ for $x \in \R$,
and $\mathbbm{1}_X(\cdot)$ stands for the indicator function of the set $X$.

\section{Entropy and Minimax Risk of Compact Hypoelliptic Pseudodifferential Operators}\label{sec:main-results}

We recall the definition of metric entropy for compact sets and the associated notion of entropy for compact linear operators.
Let $\mathcal{H}$ be a separable real Hilbert space and let $\mathcal{K} \subset \mathcal{H}$ be a compact set.
For $\varepsilon > 0$, an {$\varepsilon$-covering} of $\mathcal{K}$ is a finite set
$\{x_1,\dots,x_N\} \subset \mathcal{H}$ with the property that, for every
$x \in \mathcal{K}$, there exists an index $i \in \{1,\dots,N\}$ such that
\[
\|x - x_i\|_{\mathcal{H}} \le \varepsilon .
\]
The {$\varepsilon$-covering number} $N(\varepsilon;\mathcal{K},\|\cdot\|_{\mathcal{H}})$ is defined as the cardinality of a smallest such $\varepsilon$-covering; the {metric entropy} of $\mathcal{K}$ is the natural logarithm of the $\varepsilon$-covering number.

Given a compact linear operator $T\colon \mathcal{H} \to \mathcal{H}$, the {entropy} of $T$ is defined as the metric entropy of the closure of the image of the unit ball $\mathcal{B}_{\mathcal{H}} \subset \mathcal{H}$ under $T$, namely
\begin{equation}
\label{eq:def-entropy-operator}
    H_T(\varepsilon)
    \coloneqq
    \ln N \left(\varepsilon; \overline{T(\mathcal{B}_{\mathcal{H}})}, \|\cdot\|_{\mathcal{H}}\right).
\end{equation}
Since $T$ is compact, the set $\overline{T(\mathcal{B}_{\mathcal{H}})}$ is compact in $\mathcal{H}$, and the right-hand side of \eqref{eq:def-entropy-operator} is therefore finite for every $\varepsilon>0$.

The symbols treated in the classical literature on spectral asymptotics for
pseudo\-differential operators are predominantly drawn from the class
$\hypoposclass(\R^{2d})$ of positive-order hypoelliptic symbols;
see Theorem~\ref{thm-shubin} and the review of the standard spectral theory for
hypoelliptic pseudodifferential operators in
Appendix~\ref{sec:review-lit}. In this regime, the associated operators are necessarily
non-compact. Consequently, their entropy is not well defined and their minimax
risk does not converge to zero. 
By contrast, operators whose symbols lie in $\hyponegclass(\R^{2d})$
can be extended to compact operators on $L^2(\R^d)$; 
see, for example, \cite[Theorem~24.4]{shubin_pseudodifferential_2001}.
The results developed here build on spectral asymptotics for such compact
hypoelliptic operators, recalled in Theorem~\ref{thm:Dauge-Robert}.
\color{black}

The asymptotic formulae derived below are obtained under the assumption that the
volume function $V_\sigma$ is regularly varying at zero with negative index.
This assumption excludes highly irregular decay behavior while encompassing
the polynomial decay rates that arise naturally in hypoelliptic
pseudodifferential calculus, as well as mild logarithmic perturbations thereof. Its role is to ensure sufficient
regularity of the associated eigenvalue-counting function to allow for a precise
evaluation of the leading-order asymptotics in the integral characterizations of entropy
and minimax risk.

More concretely, regular variation of the volume function $V_\sigma$ implies
condition~(RC) in \cite[Theorem~2]{allard2025metricentropyminimaxrisk}, 
allowing one to pass from spectral information to the entropy and minimax-risk asymptotics
in \eqref{eq:main-res-entropy}--\eqref{eq:main-res-risk-2}.
Background on regular variation and a detailed proof of the implication
\textcolor{blue}{`}$V_\sigma$ regularly varying $\Rightarrow$ condition~(RC)\textcolor{blue}{'} can be found in
\cite{binghamRegularVariation1987} and in
\cite[Appendix~C]{allard2025metricentropyminimaxrisk}.


\begin{theorem}\label{thm:main-result}
Let $\sigma\in\hyponegclass(\R^{2d})$ be a strictly positive hypoelliptic symbol, and
assume that the associated volume function $V_\sigma$ is regularly varying at
zero. 
Then,
\begin{equation}\label{eq:main-res-volume-integral}
    H_\sigma(\varepsilon)
    \sim
    \int_{\varepsilon}^{\infty}
        \frac{V_\sigma(\lambda)}{\lambda}\,\diff\lambda,
    \qquad \varepsilon\to0,
\end{equation}
where $H_\sigma$ denotes the entropy of the operator associated with $\sigma$,
independently of whether $\sigma$ is quantized using the left, right, or Weyl
quantization. Moreover, if the operator $T_{\sigma^{-1}}$ is invertible, then
\begin{equation}\label{eq:main-result-part2}
    H_{T^{-1}_{\sigma^{-1}}}(\varepsilon)
    \sim
    H_\sigma(\varepsilon),
    \qquad \varepsilon\to0.
\end{equation}
\end{theorem}

\begin{proof}
See Section~\ref{sec:proof-main-result}.
\end{proof}

\noindent
The proof of Theorem~\ref{thm:main-result} is carried out within the standard
hypoelliptic pseudodifferential operator framework summarized in
Appendix~\ref{sec:review-lit}, following \cite{shubin_pseudodifferential_2001}.
The strict positivity assumption on the symbol $\sigma$ in
Theorem~\ref{thm:main-result} is not essential, but it simplifies the proof by
ensuring that the symbol is invertible. A strategy for removing this assumption
proceeds as follows. By hypoellipticity, the symbol $\sigma$ has a fixed sign
outside a sufficiently large ball in phase space; without loss of generality,
this sign may be taken to be positive. One may therefore modify $\sigma$ inside
this ball so as to make it strictly positive everywhere, without affecting the
asymptotic behavior of the associated volume function $V_\sigma$, and hence
without altering the resulting entropy asymptotics. Consequently, the additional (strict) positivity assumption in
Theorem~\ref{thm:main-result} entails no loss of generality for the problems
considered here. Moreover, all symbols arising in the applications discussed in
Section~\ref{sec:appli-sobolev-spaces} are strictly positive.

The following heuristic motivates the use of \eqref{eq:volume-ecf-relation} in the
derivation of the entropy and minimax-risk asymptotics
\eqref{eq:main-res-entropy} and \eqref{eq:main-res-risk-1}. Guided by the
uncertainty principle, one expects the eigenfunctions of $T_\sigma$ to be
essentially localized in phase space on regions of unit volume. Within such a
localized region centered at a point $(x_0,\omega_0)\in\R^{2d}$, the symbol
$\sigma$ may be regarded as approximately constant, with value
$\sigma(x_0,\omega_0)$.

Under this approximation, the number of eigenvalues exceeding a level
$\lambda>0$ is expected to be proportional to the number of disjoint unit-volume
regions of phase space centered at points $(x,\omega)$ for which
$\sigma(x,\omega)>\lambda$. This leads naturally to the volume-based relation
\eqref{eq:volume-ecf-relation}. A related heuristic for elliptic operators is
discussed by Fefferman in \cite{fefferman_uncertainty_1983}. Of course, the
argument above is purely heuristic; a central contribution of the present work is
to identify precise conditions under which the insertion of
\eqref{eq:volume-ecf-relation} into the entropy and minimax-risk integrals can be
rigorously justified in the compact, symbol-decay setting considered here.

The volume principle underlying \eqref{eq:volume-ecf-relation} is reminiscent of the
approximate diagonalization results in underspread operator theory
\cite{kailath1963timevariant, bello1969measurement, pfander2006measurement,
durisi2008noncoherent, heckel2013identification, matz1998time, matz2013time}, in that
both frameworks interpret the Weyl symbol as an effective time--frequency transfer
function and relate spectral quantities to its pointwise values through a
phase-space localization principle.
In that setting, one expects entropy formulae of the same log-integral form as \eqref{eq:main-res-entropy}, paralleling Shannon-type expressions for frequency-selective channels. Developing a rigorous analogue of Theorem~\ref{thm:main-result} for underspread operators is an interesting direction for future work.

We now present a reformulation of Theorem~\ref{thm:main-result} that makes the
connection between entropy and the symbol explicit. Again, the
symbol $\sigma$ in Corollary~\ref{cor:ME-Main} may be taken to be the left, right,
or Weyl symbol of the associated operator.

\begin{corollary}\label{cor:ME-Main}
Let $\sigma\in\hyponegclass(\R^{2d})$ be a strictly positive hypoelliptic symbol
and assume that the associated volume function $V_\sigma$ is regularly varying at
zero with negative index. Then, the entropy admits the
asymptotic representation
\begin{equation}\label{eq:main-res-entropy-corr}
    H_{\sigma}(\varepsilon)
    \sim \int_{\R^{2d}} \ln_{+}\!\left(\frac{\sigma(x,\omega)}{\varepsilon}\right)
    \diff x\,\diff\omega,
    \qquad \varepsilon \to 0.
\end{equation}
\end{corollary}



\begin{proof}
    See Section~\ref{sec:proof-ME-Main}.
\end{proof}

\noindent
Recall that, after normalization by a factor $\ln(2)$, the metric entropy of a
compact set admits an information-theoretic interpretation as the number of bits
required to uniformly encode the set with accuracy $\varepsilon$.
From this perspective, relation \eqref{eq:main-res-entropy-corr} suggests the
following interpretation.
The phase space $\R^{2d}$ may be partitioned into regions of unit volume, centered
at points $(x,\omega)$, each contributing to the total entropy the number of bits
required to quantize the local symbol value $\sigma(x,\omega)$.
At the heuristic level, this contribution is given by
$\lceil \log_{+}(\sigma(x,\omega)/\varepsilon)\rceil$, which is equivalent in the
limit $\varepsilon\to0$ to $\log_{+}(\sigma(x,\omega)/\varepsilon)$.
Corollary~\ref{cor:ME-Main} shows that this phase-space heuristic is in fact
correct and yields a sharp characterization of the leading-order term of the
entropy.


Relation \eqref{eq:main-res-entropy-corr} admits a direct phase-space
interpretation in terms of activation and deactivation of degrees of freedom.
Specifically, only those phase-space cells for which the local symbol value
$\sigma(x,\omega)$ exceeds the threshold $\varepsilon$ contribute to the entropy,
while cells with $\sigma(x,\omega)\le\varepsilon$ are effectively inactive.
For each active cell, the contribution to the entropy is logarithmic in the
ratio $\sigma(x,\omega)/\varepsilon$. From this perspective, Corollary~\ref{cor:ME-Main} shows that the overall entropy is
obtained by aggregating the contributions of all active phase-space cells.


Minimax risk has been studied extensively in nonparametric estimation for function
classes and compact subsets of Hilbert spaces, in particular through the analysis
of the spectral decay of associated compact operators, such as covariance
operators, embedding operators, or diagonal operators arising from basis
expansions; see, for example,
\cite{ibragimov1981statistical,pinsker1980optimal,tsybakovIntroductionNonparametricEstimation2009}.
In these settings, the function class is typically represented implicitly as the
image of a unit ball under a compact linear operator, and minimax risk is
characterized in terms of the eigenvalues or singular values of that operator.

In contrast, the present framework treats minimax risk as a functional of the
operator itself and aims to characterize its asymptotic behavior directly at the
symbol level. This operator-centric perspective makes it possible to link minimax
risk explicitly to phase-space geometry and spectral asymptotics, in parallel with
the corresponding entropy analysis developed above. 

To make this operator-level perspective precise, let
$T\colon \mathcal{H}\to\mathcal{H}$ be a compact linear operator on a separable
Hilbert space $\mathcal{H}$. The minimax risk of $T$ is defined over the compact
set
\(
\mathcal{K}=\overline{T(\mathcal{B}_{\mathcal{H}})}\subset\mathcal{H}.
\)
Specifically, for a noise level $\kappa>0$, it is given by
\begin{equation*}
    R_T(\kappa)
    \coloneqq
    \inf_{\hat x_\kappa}
    \sup_{x\in \overline{T(\mathcal{B}_{\mathcal{H}})}}
    \mathbb{E}_{y\sim x}\!\left[
        \bigl\|\hat x_\kappa(y)-x\bigr\|_{\mathcal{H}}^{2}
    \right],
\end{equation*}
where the observation model is
\(
y = x + \kappa \xi,
\)
with $\xi$ a Gaussian random element in $\mathcal{H}$ whose coordinates with
respect to some (and hence any) orthonormal basis of $\mathcal{H}$ are i.i.d.\
standard normal random variables. Here, $\mathbb{E}_{y\sim x}$ is expectation with respect to the law of $y$ for fixed
$x$, and $\hat x_\kappa$ is a measurable estimator based on the observation $y$.

The asymptotic behavior of the minimax risk then follows from
Theorem~\ref{thm:main-result} via the general correspondence between metric entropy
and minimax risk developed in \cite{allard2025metricentropyminimaxrisk}.

\begin{corollary}\label{cor:MR-Main}
Let $\sigma\in\hyponegclass(\R^{2d})$ be a strictly positive hypoelliptic symbol
and assume that the associated volume function $V_\sigma$ is regularly varying at
zero with negative index. Then, the minimax risk satisfies
\begin{equation}\label{eq:main-res-risk-1-corr}
    R_{\sigma}(\kappa)
    \sim
    \kappa^{2}
    \int_{\R^{2d}}
        \left(1-\frac{\varepsilon_\kappa}{\sigma(x,\omega)}\right)_{+}
        \,\diff x\,\diff\omega,
    \qquad \kappa \to 0,
\end{equation}
where the critical radius $\varepsilon_\kappa$ is determined (implicitly) by
\begin{equation}\label{eq:main-res-risk-2-corr}
    \kappa^{2}
    \int_{\R^{2d}}
        \frac{1}{\sigma(x,\omega)\,\varepsilon_\kappa}
        \left(1-\frac{\varepsilon_\kappa}{\sigma(x,\omega)}\right)_{+}
        \,\diff x\,\diff\omega
    = 1,
    \qquad \kappa>0.
\end{equation}
\end{corollary}

\begin{proof}
See Section~\ref{sec:proof-MR-Main}.
\end{proof}

\section{Applications to Sobolev Spaces}\label{sec:appli-sobolev-spaces}

We now apply the results developed in the preceding sections to characterize the
metric entropy of unit balls in Sobolev spaces on $\R^d$. This extends the
corresponding analysis for Sobolev spaces on bounded domains presented in
\cite[Section~3.2]{ab2025compactoperators}. In contrast to the bounded-domain
setting, the unit ball of a classical Sobolev space on $\R^d$—as defined, for
example, in \cite[Definition~3.1]{Taylor+1981} or
\cite[Section~9.1]{brezisFunctionalAnalysisSobolev2011}—is not compact with
respect to the $L^2(\R^d)$ metric due to translation invariance
(cf.\ \cite[Example~6.11]{adamsSobolevSpaces2003}). As a result, its metric
entropy is infinite for every $\varepsilon>0$.

Compactness-restoring variants of Sobolev spaces on $\R^d$, such as those
involving confining potentials or weights, are classical and have been studied
extensively from the perspectives of functional analysis and spectral theory.

By modifying the Sobolev structure so as to restore compactness while
preserving the underlying notion of regularity, one obtains function classes
that fall within the scope of the present theory. Two such modifications are
particularly natural and widely used in practice: Sobolev spaces with a
confining potential (Theorem~\ref{thm:sob-sch}) and weighted Sobolev spaces
(Theorem~\ref{thm:sob-weighted}). In both cases, the resulting spaces admit
compact embeddings into $L^2(\R^d)$ and can be treated directly within our framework.


We begin with Sobolev spaces equipped with a confining potential.
Fix $s>0$, let $u\colon\R^d\to\R_+^*$ be a potential function, and define
\begin{equation}\label{eq:introduce-operator-sch}
    T_{\mathcal{S}}
    =
    (I+U-\Delta)^{s/2},
\end{equation}
where $I$ is the identity operator, $U$ acts by pointwise multiplication with $u(\cdot)$,
and $-\Delta$ is the Laplacian operator.
Throughout, $T_{\mathcal{S}}$ is understood as the self-adjoint operator on
$L^2(\R^d)$ obtained as the closure of its action on $C_0^\infty(\R^d)$.
We note that $I+U-\Delta$ is a Schrödinger operator; see, for example,
\cite[Chapter~11.2]{lieb2001analysis}.

When $s$ is an even integer, say $s/2=k\in\N^*$, the operator $T_{\mathcal{S}}$
coincides with the $k$-fold composition of $I+U-\Delta$ with itself and is a
differential operator of order $2k$. Its principal Kohn--Nirenberg symbol is
\begin{equation}\label{eq:KN-symbol-sch}
    \sigma_{\mathcal{S}}(x,\omega)
    =
    \bigl(1+u(x)+(2\pi\|\omega\|_2)^2\bigr)^{s/2}.
\end{equation}
For general $s>0$, the operator $T_{\mathcal{S}}$ is defined via functional
calculus and is a pseudodifferential operator whose principal Kohn--Nirenberg
symbol is given by \eqref{eq:KN-symbol-sch}; see, for instance,
\cite[Definition~3.13]{hinzmicrolocal}. As a concrete and widely studied example, we work with the quadratic
potential (cf.\ \cite[Section~25.3]{shubin_pseudodifferential_2001}),
\[
    u(x)=c\,\|x\|_2^2,
    \qquad x\in\R^d,
\]
with a fixed constant $c>0$.

The Sobolev space of order $s$ induced by $T_{\mathcal{S}}$ is defined as
\[
\mathcal{H}_{\mathcal{S}}^{(s)}
\;\coloneqq\;
\overline{C_0^\infty(\R^d)}^{\,\|\cdot\|_{\mathcal{S}}^{(s)}},
\qquad
\|\varphi\|_{\mathcal{S}}^{(s)}
\coloneqq
\|T_{\mathcal{S}}\varphi\|_{L^2(\R^d)},
\quad \varphi\in C_0^\infty(\R^d),
\]
where the closure is taken with respect to the topology induced by
$\|\cdot\|_{\mathcal{S}}^{(s)}$, so that
$\mathcal{H}_{\mathcal{S}}^{(s)}$ embeds continuously into $L^2(\R^d)$.
Equivalently, $\mathcal{H}_{\mathcal{S}}^{(s)}$ coincides with the domain of the
operator $T_{\mathcal{S}}$ on $L^2(\R^d)$, endowed with the norm
\[
\|f\|_{\mathcal{S}}^{(s)}=\|T_{\mathcal{S}}f\|_{L^2(\R^d)}.
\]
In this sense, $\mathcal{H}_{\mathcal{S}}^{(s)}$ may be viewed as a Sobolev space
of order $s$ with a confining potential, encoding $s$ degrees of smoothness
together with spatial localization.
The corresponding unit ball is
\begin{equation}\label{eq:definition-ballsch}
    \mathcal{B}_{\mathcal{S}}^{(s)}
    \coloneqq
    \bigl\{f\in\mathcal{H}_{\mathcal{S}}^{(s)}:
    \|T_{\mathcal{S}}f\|_{L^2(\R^d)}\le 1\bigr\}.
\end{equation}

We are now in a position to characterize the metric entropy of
$\mathcal{B}_{\mathcal{S}}^{(s)}$.

\begin{theorem}\label{thm:sob-sch}
Let $d\in\N^*$ and $s,c>0$.
The metric entropy of $\mathcal{B}_{\mathcal{S}}^{(s)}$ obeys
\begin{equation}\label{eq:pinsker-entropy}
    H\!\left(\varepsilon;\mathcal{B}_{\mathcal{S}}^{(s)}\right)
    \sim
    \frac{s\,\omega_{2d}}{2d(2\pi\sqrt{c})^{d}}\,
    \varepsilon^{-2d/s},
    \qquad \varepsilon\to0,
\end{equation}
and the corresponding minimax risk admits the asymptotic expansion
\begin{equation}\label{eq:pinsker-unbounded}
    R_\kappa \left(\mathcal{B}_{\mathcal{S}}^{(s)}\right)
    \sim
    \frac{d+s}{d}
    \left(
        \frac{d\,s\,\omega_{2d}\,\kappa^2}
        {(2\pi\sqrt{c})^{d}(d+s)(2d+s)}
    \right)^{\frac{s}{d+s}},
    \qquad \kappa\to0.
\end{equation}
\end{theorem}

\begin{proof}
    See Section~\ref{sec:proof-thm-sob-sch}.
\end{proof}

\noindent
Theorem~\ref{thm:sob-sch}, in particular the minimax-risk asymptotic
\eqref{eq:pinsker-unbounded}, extends Pinsker’s theorem to Sobolev spaces on
unbounded domains. Pinsker’s original result was established in
\cite{pinsker1980optimal}; general reviews can be found in
\cite{nussbaum1999minimax,johnstone2019estimation,tsybakovIntroductionNonparametricEstimation2009},
and an extension to arbitrary bounded domains is provided in
\cite[Theorem~9]{allard2025metricentropyminimaxrisk}. The result obtained here
shows that Pinsker’s principle is not confined to bounded domains: after
restoring compactness through spatial confinement, the minimax risk remains
governed by the same structural mechanism identified by Pinsker, with the sharp
asymptotic constant now determined by phase-space volume rather than boundary
effects. This indicates that Pinsker’s theorem reflects an underlying geometric
principle of statistical estimation, rather than an artifact of bounded domains
or specific boundary conditions.


As a second application, we consider weighted Sobolev spaces, defined following
\cite[Section~3.3]{hinzmicrolocal}.
Fix $s,r,c>0$, set $u(x)=c\|x\|_2^2$ for $x\in\R^d$, and let
\begin{equation*}
    T_{\mathcal{W}}
    \coloneqq
    (I-\Delta)^{s/2}(I+U)^{r/2},
\end{equation*}
where $I$ is the identity operator, $U$ acts by pointwise multiplication with $u(\cdot)$,
and $-\Delta$ is the Laplacian operator.
As before, $T_{\mathcal{W}}$ is interpreted as a pseudodifferential operator; in
this case, its principal Kohn--Nirenberg symbol is given by
\begin{equation}\label{eq:symb-sob-weight-inverse}
    \sigma_{\mathcal{W}}(x,\omega)
    =
    \bigl(1+(2\pi\|\omega\|_2)^2\bigr)^{s/2}
    \bigl(1+c\|x\|_2^2\bigr)^{r/2}.
\end{equation}


The weighted Sobolev space of regularity $s$ and weight exponent $r$ induced by
$T_{\mathcal{W}}$ is defined as
\[
\mathcal{H}_{\mathcal{W}}^{(s,r)}
\;\coloneqq\;
\overline{C_0^\infty(\R^d)}^{\,\|\cdot\|_{\mathcal{W}}^{(s,r)}},
\qquad
\|\varphi\|_{\mathcal{W}}^{(s,r)}
\coloneqq
\|T_{\mathcal{W}}\varphi\|_{L^2(\R^d)},
\quad \varphi\in C_0^\infty(\R^d),
\]
where the closure is taken with respect to the topology induced by
$\|\cdot\|_{\mathcal{W}}^{(s,r)}$, so that
$\mathcal{H}_{\mathcal{W}}^{(s,r)}$ embeds continuously into $L^2(\R^d)$.
Equivalently, $\mathcal{H}_{\mathcal{W}}^{(s,r)}$ coincides with the domain of the
operator $T_{\mathcal{W}}$ on $L^2(\R^d)$, endowed with the norm
\[
\|f\|_{\mathcal{W}}^{(s,r)}=\|T_{\mathcal{W}}f\|_{L^2(\R^d)}.
\]
In this sense, $\mathcal{H}_{\mathcal{W}}^{(s,r)}$ may be viewed as a weighted
Sobolev space of regularity $s$, where the weight exponent $r$ enforces spatial
localization.
The corresponding unit ball is
\begin{equation}\label{eq:definition-ballweight}
    \mathcal{B}_{\mathcal{W}}^{(s,r)}
    \coloneqq
    \bigl\{f\in\mathcal{H}_{\mathcal{W}}^{(s,r)}:
    \|T_{\mathcal{W}}f\|_{L^2(\R^d)}\le 1\bigr\}.
\end{equation}

We are now ready to characterize the metric entropy of
$\mathcal{B}_{\mathcal{W}}^{(s,r)}$.

\begin{theorem}\label{thm:sob-weighted}
Let $d \in \N^*$ and $s,r,c>0$. 
The metric entropy of $\mathcal{B}_{\mathcal{W}}^{(s,r)}$ obeys
\begin{equation*}
    H \left(\varepsilon;\mathcal{B}_{\mathcal{W}}^{(s,r)}, \|\cdot\|_{L^2(\R^d)}\right)
    \sim \frac{\omega_d^2}{(2\pi\sqrt{c})^{d}}
    \begin{dcases}
        \Xi_{r,s,d}\, \varepsilon^{-\frac{d}{\min\{r,s\}}}, 
        \quad &\text{if } r\neq s,\\[.25cm]           
        \varepsilon^{-\frac{d}{s}}\ln\left(\varepsilon^{-1}\right), \quad &\text{if } r=s,
    \end{dcases}
    \quad  \varepsilon\to 0,
\end{equation*}
where, for $r \neq s$,
\begin{equation}\label{eq:expression-constant-sobw}
     \Xi_{r,s,d}
     \coloneqq \frac{\min\{r,s\} \, \Gamma\left(\frac{d \, |s-r| }{2 \min\{r,s\}}\right) \Gamma \left( \frac{d}{2} \right)}{2\Gamma \left(\frac{d\max\{r,s\} }{2\min\{r,s\} }\right)}, 
\end{equation}
and $\Gamma$ denotes Euler's Gamma function.
\end{theorem}
\color{black}


\begin{proof}
    See Section~\ref{sec:proof-sob-weighted}.
\end{proof}




\noindent
The expression \eqref{eq:expression-constant-sobw} for the constant appearing in
Theorem~\ref{thm:sob-weighted} can often be simplified by exploiting standard
identities for the Gamma function \cite{artin}. 
In particular, when the dimension is even, say $d=2k$ with $k\in\N^*$,
the Gamma factors reduce to factorials and one obtains
\[
\Xi_{r,s,2k}
=
\frac{(k-1)!\,\min\{r,s\}^{\,k+1}}
{2\prod_{j=1}^{k}\bigl(k\max\{r,s\}-j\min\{r,s\}\bigr)},
\qquad r\neq s .
\]
For example, in dimension $d=2$ this yields
\[
\Xi_{r,s,2}
=
\frac{\min\{r,s\}^2}{2\, (\max\{r,s\}-\min\{r,s\})}.
\]

\color{black}

Theorem~\ref{thm:sob-weighted} also yields sharp entropy asymptotics for
compact embeddings between weighted Sobolev spaces defined through the
operators $T_{\mathcal W}$. In particular, for parameters $s_1>s_2$ and
$r_1>r_2$, consider the inclusion map
\[
I \colon \mathcal{H}_{\mathcal W}^{(s_1,r_1)}
\longrightarrow
\mathcal{H}_{\mathcal W}^{(s_2,r_2)} ,
\]
which is compact by Rellich’s theorem
(see \cite[Theorem~2.17 and Exercise~2.12]{hinzmicrolocal}).
The special case treated in Theorem~\ref{thm:sob-weighted}
corresponds to the choice $s_2=r_2=0$. The same analysis as in the proof of Theorem~\ref{thm:sob-weighted} applies to the
embedding operator $I$ with the parameters $(r,s)$ replaced by
$(r_1-r_2,s_1-s_2)$. Consequently, 
\begin{equation*}
    H_I(\varepsilon)
    \sim \frac{\omega_d^2}{(2\pi\sqrt{c})^{d}}
    \begin{dcases}
        \Xi_{r_1-r_2,s_1-s_2,d}\,
        \varepsilon^{-\frac{d}{\min\{r_1-r_2,s_1-s_2\}}},
        & \text{if } r_1+s_2 \neq r_2+s_1,\\[.25cm]
        \varepsilon^{-\frac{d}{r_1-r_2}}\ln \left(\varepsilon^{-1}\right),
        & \text{if } r_1+s_2 = r_2+s_1,
    \end{dcases}
    \qquad \varepsilon\to 0,
\end{equation*}
where the constant $\Xi_{r,s,d}$ is defined as in
Theorem~\ref{thm:sob-weighted}. The order of the entropy numbers for such embeddings has been studied
in a more general setting—for instance for weighted Besov and
Triebel–Lizorkin spaces \cite[Theorem~4.3.2]{edmundsFunctionSpacesEntropy1996}.
However, these results do not identify the leading constants.
The present approach provides these constants explicitly in the weighted
Sobolev setting.

\color{black}

\section{Proofs}

\subsection{Proof of Theorem~\ref{thm:main-result}}\label{sec:proof-main-result}

We begin by recalling that, by definition of $\hyponegclass(\R^{2d})$
(see \eqref{eq:negative-order-def-class}), there exist $\rho\in(0,1]$ and real numbers $m_+,m_-$ with $m_+\ge m_->0$, all of which we fix throughout the proof,
such that $\symbolnot \in \hyposymbclassinv(\R^{2d})$.
Moreover, by \cite[Lemma~25.1,~1)a)]{shubin_pseudodifferential_2001},
$\symbolnot \in \hyposymbclassinv(\R^{2d})$ implies that
$\symbolnot^{-1}\in \hyposymbclass(\R^{2d})$.


To derive the entropy asymptotics, we first reduce the Weyl quantization
$T_\sigma$ to a positive self-adjoint operator.
By the polar decomposition \cite[VIII.3.11]{conway},
there exists a partial isometry $U$ such that
$T_\sigma = U |T_\sigma|$.
Since $U$ acts isometrically on the range of $|T_\sigma|$,
it maps the image of the unit ball under $|T_\sigma|$
onto the image of the unit ball under $T_\sigma$.
Consequently,
\begin{equation}\label{eq:equality-entropies-absolute-value}
    H_{T_\sigma}(\varepsilon)
    =
    H_{|T_\sigma|}(\varepsilon),
    \qquad \varepsilon>0.
\end{equation}

Our strategy is to invoke the spectral asymptotics of
Dauge and Robert stated in Theorem~\ref{thm:Dauge-Robert}.
To this end, we choose the weight functions
\begin{equation}\label{eq:weight-choice-hypo}
    \phi(z)
    = \varphi(z)
    = \bigl(1+\|z\|_2^2\bigr)^{\rho/2},
    \qquad
    w(z)=\sigma(z),
    \qquad z\in\R^{2d}.
\end{equation}
Once it is verified that this choice satisfies conditions
\emph{(H1)}, \emph{(H2)}, \emph{(W)}, and \emph{(N)}, and that
$\sigma\in\Gamma(\R^{2d};w,\phi,\varphi)$,
Theorem~\ref{thm:Dauge-Robert} yields
\begin{equation}\label{eq:spec-asymptotics}
M_{|T_\sigma|}(\lambda)
\sim
V_\sigma(\lambda),
\qquad \lambda\to0.
\end{equation}
Since $w=\sigma$, the assumption
$V_w(\lambda)=O(V_\sigma(\lambda))$ required in
Theorem~\ref{thm:Dauge-Robert} holds trivially.
Because $V_\sigma$ is assumed to be regularly varying at zero,
it follows that $M_{|T_\sigma|}$ is regularly varying at zero as well.

Since $|T_\sigma|$ is a positive compact operator,
it admits an orthonormal eigenbasis
$\{e_n\}_{n\in\N^*}$ with corresponding eigenvalues
$\{\lambda_n\}_{n\in\N^*}$ listed in nonincreasing order
and converging to zero.
In this basis, the image of the unit ball under $|T_\sigma|$
is the ellipsoid
\[
\mathcal E_\lambda
=
\Bigl\{
x=\sum_{n=1}^\infty x_n e_n :
\sum_{n=1}^\infty \frac{x_n^2}{\lambda_n^2}\le 1
\Bigr\}.
\]
The semi-axis-counting function of the ellipsoid $\mathcal E_\lambda$
coincides with the eigenvalue-counting function $M_{|T_\sigma|}(\lambda)$ of $|T_\sigma|$.
Hence, the entropy $H_{|T_\sigma|}(\varepsilon)$ is precisely the
metric entropy of the ellipsoid $\mathcal E_\lambda$.
Since $M_{|T_\sigma|}(\lambda)\sim V_\sigma(\lambda)$ as
$\lambda\to0$ and $V_\sigma$ is regularly varying at zero,
the regularity condition \emph{(RC)} required in
\cite[Theorem~2]{allard2025metricentropyminimaxrisk}
is satisfied (cf. the discussion following \cite[Lemma~10]{allard2025metricentropyminimaxrisk}).
We may therefore apply \cite[Theorem~2]{allard2025metricentropyminimaxrisk}
to conclude that
\[
H_{|T_\sigma|}(\varepsilon)
\sim
\int_{\varepsilon}^{\infty}
\frac{M_{|T_\sigma|}(\lambda)}{\lambda}\,d\lambda,
\qquad \varepsilon\to0.
\]
Combining this with \eqref{eq:equality-entropies-absolute-value} and \eqref{eq:spec-asymptotics} yields
\begin{equation}\label{eq:mainr-result-for-Weyl-1}
H_{T_\sigma}(\varepsilon)
=
H_{|T_\sigma|}(\varepsilon)
\sim
\int_{\varepsilon}^{\infty}
\frac{V_\sigma(\lambda)}{\lambda}\,d\lambda,
\qquad \varepsilon\to0,
\end{equation}
which is precisely \eqref{eq:main-res-volume-integral}.
It remains to verify conditions \emph{(H1)}–\emph{(N)}
for the choice \eqref{eq:weight-choice-hypo}.

\color{black}
\textbf{Verifying (H1).}
To establish that $\phi^{-1}$ and $\varphi^{-1}$ are $(\phi,\varphi)$-continuous,
we start from the elementary inequality
\[
1+(t_1+t_2)^2 \le \bigl((1+t_1^2)^{1/2}+t_2\bigr)^2,
\qquad t_1,t_2\ge 0,
\]
which follows by direct expansion. Applying this with
$t_1=\|z\|_2$ and $t_2=\|z'\|_2$, and using the triangle inequality
$\|z+z'\|_2 \le \|z\|_2+\|z'\|_2$, we obtain
\begin{equation}\label{eq:verifyH1-DR-cont-1}
\bigl(1+\|z+z'\|_2^2\bigr)^{1/2}
\le
\bigl(1+\|z\|_2^2\bigr)^{1/2}+\|z'\|_2,
\qquad z,z'\in\R^{2d}.
\end{equation}
Replacing \(z\) by \(z+z'\) and \(z'\) by \(-z'\) in
\eqref{eq:verifyH1-DR-cont-1}, yields the
corresponding lower bound
\begin{equation}\label{eq:verifyH1-DR-cont-2}
\bigl(1+\|z+z'\|_2^2\bigr)^{1/2}
\ge
\bigl(1+\|z\|_2^2\bigr)^{1/2}-\|z'\|_2,
\qquad z,z'\in\R^{2d}.
\end{equation}
Fix $c>0$ and take $x,x',\omega,\omega'\in\R^d$ such that
\[
\frac{\|x'\|_2}{\varphi(x,\omega)}+\frac{\|\omega'\|_2}{\phi(x,\omega)}\le c.
\]
Setting $z=(x,\omega)$ and $z'=(x',\omega')$, and using $\phi(z)=\varphi(z)=(1+\|z\|_2^2)^{\rho/2}$, we obtain
\[
\|z'\|_2
\le \|x'\|_2+\|\omega'\|_2
\le c\,(1+\|z\|_2^2)^{\rho/2}
\le c\,(1+\|z\|_2^2)^{1/2},
\]
since $\rho\le 1$. Combining this with
\eqref{eq:verifyH1-DR-cont-1}–\eqref{eq:verifyH1-DR-cont-2} yields
\[
1-c
\le
\frac{(1+\|z+z'\|_2^2)^{1/2}}{(1+\|z\|_2^2)^{1/2}}
\le
1+c.
\]
Raising both sides to the power $\rho$ and choosing $c=1/2$ shows that
\[
C^{-1}
\le
\frac{\phi^{-1}(z+z')}{\phi^{-1}(z)}
\le
C,
\qquad C=2^\rho,
\]
which proves that $\phi^{-1}$ is $(\phi,\varphi)$-continuous. Since $\phi=\varphi$ by \eqref{eq:weight-choice-hypo}, we have $\phi^{-1}=\varphi^{-1}$.
Therefore, the same estimate holds for $\varphi^{-1}$, and both functions are
$(\phi,\varphi)$-continuous.

We next verify that $\phi^{-1}$ and $\varphi^{-1}$ are $(1,1)$-temperate.
By Peetre’s inequality (see \cite[(2.21)]{treves_introduction_1980}),
\begin{equation}\label{eq:Peetre}
(1+\|z\|_2^2)^{1/2}
\le
\sqrt{2}\,(1+\|z'\|_2^2)^{1/2}(1+\|z+z'\|_2^2)^{1/2},
\qquad z,z'\in\R^{2d}.
\end{equation}
Rewriting \eqref{eq:Peetre} gives
\begin{equation}\label{eq:DR-HA-hypo-tempstep}
(1+\|z+z'\|_2^2)^{-\rho/2}
\le
2^{\rho/2}\,
(1+\|z\|_2^2)^{-\rho/2}
(1+\|z'\|_2^2)^{\rho/2}.
\end{equation}
For $z=(x,\omega)$ and $z'=(x',\omega')$, using
$1+\|z'\|_2^2\le (1+\|x'\|_2+\|\omega'\|_2)^2$, we obtain
\[
\phi^{-1}(x+x',\omega+\omega')
\le
2^{\rho/2}\,\phi^{-1}(x,\omega)\,
\bigl(1+\|x'\|_2+\|\omega'\|_2\bigr)^\rho.
\]
Since $\phi=\varphi$ by \eqref{eq:weight-choice-hypo}, we have $\phi^{-1}=\varphi^{-1}$.
Hence the same bound holds for $\varphi^{-1}$, and both functions are
$(1,1)$-temperate.

Finally, since $\rho>0$, we have $\phi^{-1}(z)=\varphi^{-1}(z)\le 1$ for all
$z\in\R^{2d}$, and hence both functions are bounded on $\R^{2d}$.
Together with the $(\phi,\varphi)$-continuity and $(1,1)$-temperateness
established above, this verifies condition \emph{(H1)}.

\textbf{Verifying (H2).}
To verify condition \emph{(H2)}, we start from the elementary bound $2t \le 1+t^2$, for $t\ge 0$.
Applying this estimate first with $t=\|x\|_2$ and then with $t=\|\omega\|_2$, where
$x,\omega\in\R^d$, and adding the resulting inequalities, yields
\[
1+\|x\|_2+\|\omega\|_2
\le
2\bigl(1+\|x\|_2^2+\|\omega\|_2^2\bigr),
\qquad x,\omega\in\R^d.
\]
Recalling that $\phi(x,\omega)=\varphi(x,\omega)
=(1+\|x\|_2^2+\|\omega\|_2^2)^{\rho/2}$, this bound can be rewritten as
\[
1+\|x\|_2+\|\omega\|_2
\le
2\bigl(\phi(x,\omega)\,\varphi(x,\omega)\bigr)^{1/\rho},
\qquad x,\omega\in\R^d.
\]
This is exactly condition~\eqref{eq:DR-H2} with $C=2^{-\rho}$ and $\zeta=\rho$.


\textbf{Verifying (W).}
By hypoellipticity of $\sigma$ (recall \eqref{eq:bound-def-hypoellip}), there exists
$R>0$ such that, for every multi-index $\alpha\in\N^{2d}$, the bound
\begin{equation}\label{eq:verify-W-hypo-DR-11}
    \bigl|\partial^\alpha \sigma(z)\bigr|
    \le
    C_\alpha\,\sigma(z)
    \bigl(1+\|z\|_2^2\bigr)^{-\frac{\rho|\alpha|}{2}}
\end{equation}
holds for all $z\in\R^{2d}$ with $\|z\|_2\ge R$, for some constant $C_\alpha>0$
depending on $\alpha$.
For fixed $\alpha$, the function
$z\mapsto \sigma(z)\bigl(1+\|z\|_2^2\bigr)^{-\frac{\rho|\alpha|}{2}}$ is continuous
and strictly positive, while $z\mapsto \partial^\alpha\sigma(z)$ is continuous
since $\sigma\in C^\infty(\R^{2d})$.
As the closed ball $\{z\in\R^{2d}:\|z\|_2\le R\}$ is compact, there exists a
constant $C'_\alpha>0$ such that
\begin{equation}\label{eq:verify-W-hypo-DR-12}
    \bigl|\partial^\alpha \sigma(z)\bigr|
    \le
    C'_\alpha\,\sigma(z)
    \bigl(1+\|z\|_2^2\bigr)^{-\frac{\rho|\alpha|}{2}},
    \qquad
    \|z\|_2 \le R.
\end{equation}
Combining \eqref{eq:verify-W-hypo-DR-11} and \eqref{eq:verify-W-hypo-DR-12}, and
recalling that $w=\sigma$ and $\phi=\varphi=(1+\|\cdot\|_2^2)^{\rho/2}$, we obtain,
for all $x,\omega\in\R^d$ and all multi-indices $\alpha,\beta\in\N^d$,
\begin{align}
    \bigl|\partial_x^\alpha \partial_\omega^\beta w(x,\omega)\bigr|
    &\le
    C_{\alpha,\beta}\,w(x,\omega)
    \bigl(1+\|(x,\omega)\|_2^2\bigr)^{-\frac{\rho(|\alpha|+|\beta|)}{2}}
    \label{eq:verify-W-hypo-1}\\
    &=
    C_{\alpha,\beta}
        \frac{w(x,\omega)}
        {\phi(x,\omega)^{|\alpha|}\,\varphi(x,\omega)^{|\beta|}}
    .
    \label{eq:verify-W-hypo-2}
\end{align}
This establishes condition~\eqref{eq:weight-class-DR} and hence shows that
$w\in\Gamma(\R^{2d};w,\phi,\varphi)$, as required for \emph{(W)}.

We show that the weight $w=\sigma$ is $(\phi,\varphi)$-temperate by a case
distinction on $z'$.
Fix $z,z'\in\R^{2d}$.

\medskip
\noindent\emph{Case 1:} 
    $\|z'\|_2 \leq \|z\|_2^\rho /2$.
Define the function
\[
    g \colon [0,1] \to \R,
    \qquad
    g(t) \coloneqq \ln\bigl(\sigma(z+t z')\bigr).
\]
Then
\begin{equation}\label{eq:proof-weight-temperate-DR1}
    g'(t)
    =
    \frac{\nabla \sigma(z+t z') \cdot z'}{\sigma(z+t z')},
    \qquad t\in[0,1].
\end{equation}
By hypoellipticity of $\sigma$, applying
\eqref{eq:verify-W-hypo-DR-11}--\eqref{eq:verify-W-hypo-DR-12} with $|\alpha|=1$
yields bounds on each first-order partial derivative of $\sigma$.
Combining the componentwise bounds corresponding to $|\alpha|=1$, there exists a
constant $K>0$ such that
\begin{equation}\label{eq:proof-weight-temperate-DR2}
    \|\nabla \sigma(z+t z')\|_2
    \le
    K\,\sigma(z+t z')
    \bigl(1+\|z+t z'\|_2^2\bigr)^{-\rho/2},
    \qquad t\in[0,1].
\end{equation}
Substituting this estimate into \eqref{eq:proof-weight-temperate-DR1} and applying
the Cauchy-Schwarz inequality yields
\begin{equation}\label{eq:proof-weight-temperate-DR3}
    |g'(t)|
    \le
    K\,\|z'\|_2\,\bigl(1+\|z+t z'\|_2^2\bigr)^{-\rho/2},
    \qquad t\in[0,1].
\end{equation}
Under the restriction of this case, $\|z'\|_2 \le \|z\|_2^\rho/2$, and assuming
$\|z\|_2\ge 1$, we estimate
\[
\|z+t z'\|_2
\ge
\|z\|_2 - t\|z'\|_2
\ge
\|z\|_2 - \tfrac12 \|z\|_2^\rho
\ge
\tfrac12 \|z\|_2,
\qquad t\in[0,1],
\]
where the last inequality uses $\rho\le 1$.
Inserting this bound into \eqref{eq:proof-weight-temperate-DR3} shows that the
right-hand side is uniformly bounded for $t\in[0,1]$. Hence, there exists a
constant $K_0>0$ such that
\begin{equation}\label{eq:g-uniform-bound}
    |g'(t)| \le K_0,
    \qquad t\in[0,1].
\end{equation}
If $\|z\|_2 < 1$, then the assumption of this case implies
\(
\|z'\|_2 \le \tfrac12.
\)
Consequently, for all $t\in[0,1]$, the points $z+t z'$ remain in a fixed compact
subset of $\R^{2d}$.
Since $\sigma\in C^\infty(\R^{2d})$ is strictly positive, the function
$\nabla\sigma/\sigma$ is continuous and therefore bounded on this set.
It follows that $|g'(t)|$ is uniformly bounded for all $t\in[0,1]$, 
which yields \eqref{eq:g-uniform-bound}.
Integrating \eqref{eq:g-uniform-bound}
over the interval $[0,1]$ 
gives
\begin{equation}\label{eq:proof-weight-temperate-DR4}
    \left|\ln \left(\frac{\sigma(z+z')}{\sigma(z)}\right)\right|
    \le
    K_0,
\end{equation}
and therefore
\[
    \sigma(z+z') \le e^{K_0}\,\sigma(z).
\]
Writing \(z=(x,\omega)\) and \(z'=(x',\omega')\), and recalling that \(w=\sigma\), we hence get 
\[
w(x+x',\omega+\omega')
\le
e^{K_0}\,w(x,\omega)
\le
e^{K_0}\,w(x,\omega)\,
\bigl[1+\|x'\|_2\,\phi(x,\omega)+\|\omega'\|_2\,\varphi(x,\omega)\bigr],
\]
which is exactly \eqref{eq:def-phi-phi-temperate} with \(C=e^{K_0}\) and \(\zeta=1\).

\noindent\emph{Case 2:} $\|z'\|_2 \ge \|z\|_2^{\rho}/2$. By hypoellipticity, there exist constants $R>0$ and $C_1>0$ 
with
\[
\sigma(z)\ge C_1\,\bigl(1+\|z\|_2^2\bigr)^{-m_+/2},
\qquad \|z\|_2\ge R.
\]
Since $\sigma$ is continuous and strictly positive on $\R^{2d}$, it attains a
strictly positive minimum on the closed ball
$B_R=\{z\in\R^{2d}:\|z\|_2\le R\}$.
Hence, there exists $C_2>0$ satisfying
\[
\sigma(z)\ge C_2\,\bigl(1+\|z\|_2^2\bigr)^{-m_+/2},
\qquad \|z\|_2\le R.
\]
These estimates imply that there exists a constant $K_1>0$ such that
\[
\sigma(z)\ge K_1\,\bigl(1+\|z\|_2^2\bigr)^{-m_+/2},
\qquad z\in\R^{2d}.
\]
Using this lower bound in conjunction with the corresponding upper bound from
\eqref{eq:bound-def-hypoellip}, we obtain
\[
    \frac{\sigma(z+z')}{\sigma(z)}
    \le
   K_2\,
    \frac{\bigl(1+\|z+z'\|_2^2\bigr)^{-m_-/2}}
         {\bigl(1+\|z\|_2^2\bigr)^{-m_+/2}},
\]
for some constant $K_2>0$.
Applying Peetre’s inequality \eqref{eq:Peetre} to the numerator gives
\[
    \frac{\sigma(z+z')}{\sigma(z)}
    \le
    K_3
    \bigl(1+\|z'\|_2^2\bigr)^{m_-/2}
    \bigl(1+\|z\|_2^2\bigr)^{(m_+-m_-)/2},
\]
for some constant $K_3>0$. Using the assumption of this case,
$\|z'\|_2 \ge \|z\|_2^{\rho}/2$, we can bound
$\bigl(1+\|z\|_2^2\bigr)^{(m_+-m_-)/2}
\le
C\,\bigl(1+\|z'\|_2^2\bigr)^{(m_+-m_-)/(2\rho)}$,
which leads to
\begin{equation}\label{eq:proof-weight-temperate-DR5}
    \sigma(z+z')
    \le
    K_4\,\sigma(z)
    \bigl(1+\|z'\|_2^2\bigr)^{\frac{m_-}{2}+\frac{m_+-m_-}{2\rho}},
\end{equation}
for some constant $K_4>0$. Writing $z=(x,\omega)$ and $z'=(x',\omega')$, we note that
\[
    1+\|z'\|_2^2
    =
    1+\|x'\|_2^2+\|\omega'\|_2^2
    \le
    \bigl(1+\|x'\|_2+\|\omega'\|_2\bigr)^2
    \le
    \bigl(1+\|x'\|_2\,\phi(x,\omega)
           +\|\omega'\|_2\,\varphi(x,\omega)\bigr)^2 .
\]
Recalling that $w=\sigma$, this implies
\[
    w(x+x',\omega+\omega')
    \le
    K_4\,w(x,\omega)
    \bigl[1+\|x'\|_2\,\phi(x,\omega)
          +\|\omega'\|_2\,\varphi(x,\omega)\bigr]^{\,m_-+\frac{m_+-m_-}{\rho}},
\]
which is precisely the $(\phi,\varphi)$-temperateness condition
\eqref{eq:def-phi-phi-temperate} with
$C=K_4$ and $\zeta = m_- + (m_+-m_-)/\rho$.

\textbf{Verifying (N).}
By hypoellipticity of $\sigma\in\hyposymbclassinv(\R^{2d})$
(recall \eqref{eq:bound-def-hypoellip}),
there exist constants $C_1,C_2,R>0$ such that
\begin{equation}\label{eq:verify-N-DR-hypo}
    C_1\bigl(1+\|z\|_2^2\bigr)^{-m_{+}/2}
    \le
    \sigma(z)
    \le
    C_2\bigl(1+\|z\|_2^2\bigr)^{-m_{-}/2},
    \qquad
    z\in\R^{2d}\ \text{with}\ \|z\|_2\ge R .
\end{equation}
With the choice $w=\sigma$, the bounds in \eqref{eq:verify-N-DR-hypo} imply
\begin{equation}\label{eq:verify-N-DR-hypo2}
    C_2^{-\frac{2\rho}{m_-}}\, w(z)^{\frac{2\rho}{m_-}}
    \le
    (1+\|z\|_2^2)^{-\rho}
    \le
    C_1^{-\frac{2\rho}{m_+}}\, w(z)^{\frac{2\rho}{m_+}},
    \qquad
    z\in\R^{2d}\ \text{with}\ \|z\|_2\ge R .
\end{equation}
Since $w=\sigma$ is continuous and strictly positive, its restriction to the
closed ball $\{z\in\R^{2d}:\|z\|_2\le R\}$ attains a positive minimum and a finite
maximum. Define 
\[
    C_- \coloneqq \min_{\|z\|_2\le R} w(z),
    \qquad
    C_+ \coloneqq \max_{\|z\|_2\le R} w(z).
\]
Then, for all $z\in\R^{2d}$ with $\|z\|_2\le R$,
\begin{equation}\label{eq:verify-N-DR-hypo3}
    \left(1+R^2\right)^{-\rho}  C_+^{-\frac{2\rho}{m_-}}  w(z)^{\frac{2\rho}{m_-}}
    \leq \left(1+R^2\right)^{-\rho}
    \leq \left(1+\|z\|_2^2\right)^{-\rho} 
    \leq 1 
    \leq C_-^{-\frac{2\rho}{m_+}} w(z)^{\frac{2\rho}{m_+}} .
\end{equation}
Combining \eqref{eq:verify-N-DR-hypo2} and \eqref{eq:verify-N-DR-hypo3}, and setting
\[
    K' \coloneqq
    \min \left\{
        C_2^{-\frac{2\rho}{m_-}},
        (1+R^2)^{-\rho} C_+^{-\frac{2\rho}{m_-}}
    \right\},
    \qquad
    K \coloneqq
    \max \left\{
        C_1^{-\frac{2\rho}{m_+}},
        C_-^{-\frac{2\rho}{m_+}}
    \right\},
\]
$\gamma'=2\rho/m_-$, and $\gamma=2\rho/m_+$,
we obtain
\[
    K' w(z)^{\gamma'}
    \le
    \phi^{-1}(z)\,\varphi^{-1}(z)
    \le
    K w(z)^{\gamma},
    \qquad z\in\R^{2d},
\]
which is exactly condition~\eqref{eq:DR-N}.

\color{black}



This completes the verification of conditions \emph{(H1)}–\emph{(N)} and hence
establishes \eqref{eq:main-res-volume-integral} for the Weyl quantization.
We next show that the same entropy asymptotics hold for the left and right
quantizations.
To this end, we first state a technical lemma, which is proved at the end of this
section.

\begin{lemma}\label{lem:technical-lemma}
Let $\rho_1>0$ and $m>0$.
Let $T_1$ and $T_2$ be compact linear operators on $L^2(\R^d)$ such that
\[
T_1 = T_2 (I + K) + K_{-\infty},
\]
for some $K\in \Psi^{-m}_{\rho_1}(\R^d)$ and $K_{-\infty}\in\Psi^{-\infty}(\R^d)$.
Assume further that there exists a strictly positive hypoelliptic symbol
$\sigma\in\hyponegclass(\R^{2d})$ whose associated volume function $V_\sigma$ is
regularly varying at zero, 
and that either $T_1$ or $T_2$ is the Weyl quantization
of $\sigma$.
Then 
\begin{equation}\label{eq:technical-lemma-result}
    H_{T_1}(\varepsilon)
    \sim
    H_{T_2}(\varepsilon),
    \qquad \varepsilon \to 0.
\end{equation}
\end{lemma}



\textbf{Invariance under change of quantization.}
We establish invariance of the entropy asymptotics under a change of
quantization by treating explicitly the Weyl–left case; the Weyl–right case
follows by the same arguments. Let $T_1$ and $T_2$ denote the left and Weyl quantizations of the symbol $\sigma$,
respectively, and let $\sigma_1$ be the Weyl symbol of $T_1$.
Since $T_1$ is the left quantization of a symbol
$\sigma\in\hyposymbclassinv(\R^{2d})$,
Proposition~25.1 in \cite{shubin_pseudodifferential_2001} implies that
the corresponding Weyl symbol $\sigma_1$ also belongs to
$\hyposymbclassinv(\R^{2d})$.

Since $T_2$ is the Weyl quantization of a hypoelliptic symbol in
$\hyposymbclassinv(\R^{2d})$, Theorem~25.1 in
\cite{shubin_pseudodifferential_2001} guarantees the existence of a right
parametrix $B_2$ with symbol $b_2\in\hyposymbclass(\R^{2d})$ such that
\begin{equation}\label{eq:parametrix-B-Tsigma}
    T_2 B_2 = I + S,
    \qquad S\in\Psi^{-\infty}(\R^d).
\end{equation}

We may assume, without loss of generality, that the symbol $b_2$ is nowhere vanishing on $\R^{2d}$. 
Indeed, by hypoellipticity and the lower bound in \eqref{eq:bound-def-hypoellip}, there exists $R>0$ such that
$b_2(z)\neq 0$ for all $z \in \R^{2d}$ with $\|z\|_2 \ge R$. If $b_2$ vanishes at some point in the compact region
$\{z \in \R^{2d} \mid \|z\|_2 < R \}$, we can construct a function
$b_2' \in C^\infty(\R^{2d})$, supported in
$\{\, z \in \R^{2d} \mid \|z\|_2 \le 2R \,\}$, such that
\[
b_2'' \coloneqq b_2 + b_2' \in \hyposymbclass(\R^{2d})
\]
and $b_2''(z) \neq 0$ for all $z \in \R^{2d}$. Let $B_2'$ and $B_2''$ denote the Weyl quantizations of
$b_2'$ and $b_2''$, respectively. Replacing $B_2$ by $B_2''$ in \eqref{eq:parametrix-B-Tsigma}
yields
\[
T_2 B_2'' = I + S + T_2 B_2',
\]
where $T_2 B_2' \in \Psi^{-\infty}(\R^d)$ since $b_2'$ is compactly supported. Absorbing this term
into the remainder, we obtain a representation of the form
\[
T_2 B_2'' = I + S',
\qquad S' \in \Psi^{-\infty}(\R^d),
\]
with a parametrix whose symbol does not vanish on $\R^{2d}$.
Composing \eqref{eq:parametrix-B-Tsigma} on the right with $(T_1 - T_2)$ yields
\[
T_2 B_2 (T_1-T_2) = (T_1-T_2) + S(T_1-T_2).
\]
Because $S\in\Psi^{-\infty}(\R^d)$ and $T_1-T_2$ is a pseudodifferential operator,
and since $\Psi^{-\infty}(\R^d)$ is a two-sided ideal in the
pseudodifferential calculus, the composition $S(T_1-T_2)$ belongs to
$\Psi^{-\infty}(\R^d)$; see, for example, \cite[Theorem~4.22]{hinzmicrolocal} or \cite[Theorem~23.6]{shubin_pseudodifferential_2001}.
Letting
\[
S' \coloneqq -\,S(T_1-T_2)\in\Psi^{-\infty}(\R^d),
\]
we may therefore rewrite the above identity as
\[
T_2 B_2 (T_1-T_2) = (T_1-T_2) - S'.
\]
Rearranging terms yields
\begin{equation}\label{eq:t-perturb}
T_1
=
T_2\bigl(I + B_2(T_1-T_2)\bigr) + S'.
\end{equation}
With $K \coloneqq B_2(T_1-T_2)$ and $K_{-\infty}=S'$, we see that
\eqref{eq:t-perturb} is of the form required in
Lemma~\ref{lem:technical-lemma}.
Therefore, once it is shown that 
$K\in\Psi^{-m}_{\rho_1}(\R^d)$ for some $\rho_1>0$ and $m>0$,
the lemma applies and yields
\begin{equation}\label{eq:equiv-invariance-quant}
    H_{T_1}(\varepsilon)\sim H_{T_2}(\varepsilon),
    \qquad \varepsilon\to0.
\end{equation}

We next prove that in fact $K\in\Psi^{-2\rho}_{\rho}(\R^d)$.
This will be done by first showing that $\sigma^{-1}(\sigma_1-\sigma)\in \Gamma^{-2\rho}_{\rho}(\R^{2d})$. 
Recall that $\sigma_1$ denotes the Weyl symbol of the left quantization $T_1$.
By \cite[Theorem~23.3]{shubin_pseudodifferential_2001}, for every $N\in\N^*$ there
exists a remainder
$r_N\in\Gamma^{-m_- - 2(N+1)\rho}_{\rho}(\R^{2d})$ such that
\begin{equation}\label{eq:asymp-expansion-diff-symbols}
    \sigma_1(x,\omega)-\sigma(x,\omega)
    =
    \sum_{\substack{\alpha\in\N^d\\1\le|\alpha|\le N}}
    \frac{1}{\alpha!\,2^{|\alpha|}}
    \partial_\omega^\alpha\partial_x^\alpha\sigma(x,\omega)
    + r_N(x,\omega),
    \qquad x,\omega\in\R^d.
\end{equation}
Since $\sigma$ is strictly positive, division by $\sigma$ is well-defined, and
\eqref{eq:asymp-expansion-diff-symbols} yields
\begin{equation}\label{eq:finite-decom-Weyl-left}
    \frac{\sigma_1-\sigma}{\sigma}
    =
    \sum_{\substack{\alpha\in\N^d\\1\le|\alpha|\le N}}
    \frac{1}{\alpha!\,2^{|\alpha|}}
    \frac{\partial_\omega^\alpha\partial_x^\alpha\sigma}{\sigma}
    + \frac{r_N}{\sigma}.
\end{equation}
We treat the two terms on the right-hand side of \eqref{eq:finite-decom-Weyl-left} separately.
First, since $\sigma^{-1}\in\hyposymbclass(\R^{2d})$, application of
\cite[Lemma~25.1]{shubin_pseudodifferential_2001} allows us to conclude that
\[
\frac{r_N}{\sigma}
\in
\Gamma^{\,m_+ - m_- - 2(N+1)\rho}_{\rho}(\R^{2d}).
\]
Choosing $N$ sufficiently large ensures
\begin{equation}\label{eq:est-1}
\frac{r_N}{\sigma}\in \Gamma^{-2\rho}_{\rho}(\R^{2d}).
\end{equation}
Second, consider the derivative terms appearing in the finite sum in
\eqref{eq:finite-decom-Weyl-left}. For each multi-index $\alpha\neq 0$, the
Leibniz rule together with the derivative bounds in
\eqref{eq:bound-def-hypoellip} imply that
\[
\sigma^{-1}\,\partial_\omega^\alpha\partial_x^\alpha\sigma
\in
\Gamma^{-2\rho|\alpha|}_{\rho}(\R^{2d}).
\]
Since $|\alpha|\ge 1$, we have $-2\rho|\alpha|\le -2\rho$, and hence
\begin{equation}\label{eq:est-2}
\sigma^{-1}\,\partial_\omega^\alpha\partial_x^\alpha\sigma
\in
\Gamma^{-2\rho}_{\rho}(\R^{2d}).
\end{equation}
Combining \eqref{eq:est-1} and \eqref{eq:est-2} in \eqref{eq:finite-decom-Weyl-left} yields
\[
\sigma^{-1}(\sigma_1-\sigma)\in \Gamma^{-2\rho}_{\rho}(\R^{2d}),
\]
as claimed.

Returning to the parametrix identity \(T_2 B_2 = I + S\), let \(s\) denote the Weyl
symbol of the 
operator \(S\).
By the composition formula for Weyl quantization
\cite[Theorem~23.6]{shubin_pseudodifferential_2001} and the hypoellipticity of both
\(b_2\) and \(\sigma\), for every \(M\in\N\) there exist constants
\(C_{\alpha,\beta}>0\) and a remainder
\(r'_M\in\Gamma^{m_+ - m_- - 2(M+1)\rho}_{\rho}(\R^{2d})\) such that
\begin{align}
    1+s(z)
    &=
    \sum_{\substack{\alpha,\beta\in\N^d\\0\le|\alpha|+|\beta|\le M}}
    C_{\alpha,\beta}\,
    \partial_z^\alpha b_2(z)\,
    \partial_z^\beta \sigma(z)
    + r'_M(z)
    \label{eq:long-formula-1}\\
    &=
    b_2(z)\,\sigma(z)
    \left(
        1
        + \sum_{\substack{\alpha,\beta\in\N^d\\1\le|\alpha|+|\beta|\le M}}
        C_{\alpha,\beta}\,
        \frac{\partial_z^\alpha b_2(z)}{b_2(z)}\,
        \frac{\partial_z^\beta \sigma(z)}{\sigma(z)}
        + \frac{r'_M(z)}{b_2(z)\,\sigma(z)}
    \right),
    \label{eq:long-formula-2}
\end{align}
for all \(z\in\R^{2d}\). 
Recall that $\sigma(z)$ is non-vanishing for all $z\in\R^{2d}$ by assumption, 
and that $b_2(z)$ has been chosen to be non-vanishing on $\R^{2d}$ without loss of generality.

Choosing $M$ sufficiently large ensures that $r'_M\in\Gamma^{-2\rho}_{\rho}(\R^{2d})$.
Moreover, by the Leibniz rule and hypoellipticity,
\[
b_2^{-1}\,\partial^\alpha b_2 \in \Gamma^{-\rho|\alpha|}_{\rho}(\R^{2d}),
\qquad
\sigma^{-1}\,\partial^\beta \sigma \in \Gamma^{-\rho|\beta|}_{\rho}(\R^{2d}),
\]
for all multi-indices $\alpha,\beta$.
Consequently, the expression in parentheses in \eqref{eq:long-formula-2}
is of the form $1+r'$ with $r'\in\Gamma^{-2\rho}_{\rho}(\R^{2d})$.
Since $r'(z)\to 0$ as $\|z\|_2\to\infty$, there exists $R>0$ such that
\[
|r'(z)|\le \tfrac12,
\qquad \text{for all } z\in\R^{2d} \text{ with } \|z\|_2 \ge R.
\]
We may now modify $r'$ on $\{z\in\R^{2d} \mid \|z\|_2<R \}$ as follows.
Let $\chi\in C_c^\infty(\R^{2d})$ satisfy
\[
\chi(z)=1 \quad \text{for } z\in\{z\in\R^{2d} \mid \|z\|_2\le R \},
\]
and
\[
\chi(z)=0 \quad \text{for } z\in\{z\in\R^{2d} \mid \|z\|_2\ge 2R \},
\]
and let $G\in C^\infty(\R)$ be a smooth truncation of the identity satisfying
\[
G(t)=t \quad \text{for } |t|\le \tfrac12,
\qquad
|G(t)|\le \tfrac12 \quad \text{for all } t\in\R.
\]
Define
\[
\widetilde r'(z)\coloneqq (1-\chi(z))\,r'(z)+\chi(z)\,G(r'(z)),
\qquad z\in\R^{2d}.
\]
Then
\[
\widetilde r' - r'
=
\chi\bigl(G(r')-r'\bigr)
\in C_c^\infty(\R^{2d}),
\]
and hence
\[
\widetilde r' \in \Gamma^{-2\rho}_{\rho}(\R^{2d}),
\]
since $r'\in \Gamma^{-2\rho}_{\rho}(\R^{2d})$ and
$C_c^\infty(\R^{2d})\subset \Gamma^{-\infty}(\R^{2d})
\subset \Gamma^{-2\rho}_{\rho}(\R^{2d})$.
Redefining
\[
1+\widetilde s \coloneqq b_2\,\sigma(1+\widetilde r'),
\]
we get 
\[
\widetilde s - s = b_2 \, \sigma(\widetilde r' - r')\in \Gamma^{-\infty}(\R^{2d}),
\]
since $\widetilde r'-r'\in C_c^\infty(\R^{2d})$.
Hence $\widetilde s\in\Gamma^{-\infty}(\R^{2d})$, while
$\widetilde r'\in\Gamma^{-2\rho}_{\rho}(\R^{2d})$ by construction, and the identity
\[
1+\widetilde s = b_2\,\sigma(1+\widetilde r')
\]
retains the form of \eqref{eq:long-formula-1}--\eqref{eq:long-formula-2}.
Replacing $(s,r')$ by $(\widetilde s,\widetilde r')$ in \eqref{eq:long-formula-1}--\eqref{eq:long-formula-2}, we thus preserve all structural
properties required in what follows—namely the identity, the symbol class of $r'$, and
the property of $s$—and we may therefore assume, without loss of generality, 
\[
|r'(z)|\le \tfrac12,
\qquad z\in\R^{2d}.
\]
It follows that
\[
b_2(z)=\sigma(z)^{-1}\,\frac{1+s(z)}{1+r'(z)},
\qquad z\in\R^{2d}.
\]
Moreover,
\[
(1+r')^{-1}=1+q,
\qquad
q=-\frac{r'}{1+r'}.
\]
Writing $q=r'\,H(r')$ with $H(t)\coloneqq -\frac{1}{1+t}$, and noting that $H$ is $C^\infty$ on an open neighborhood of
$[-\tfrac12,\tfrac12]$, we obtain $H(r')\in\Gamma^0_{\rho}(\R^{2d})$ and hence
\[
q\in\Gamma^{-2\rho}_{\rho}(\R^{2d}).
\]
Therefore,
\[
\frac{1+s}{1+r'} = (1+s)(1+q) = 1+\tilde r,
\qquad
\tilde r \in \Gamma^{-2\rho}_{\rho}(\R^{2d}),
\]
and consequently
\begin{equation}\label{eq:b2-symbol}
b_2=\sigma^{-1}(1+\tilde r),
\qquad \tilde r\in\Gamma^{-2\rho}_{\rho}(\R^{2d}).
\end{equation}

Applying the same arguments to the operator
\(B_2(T_1-T_2)\) shows that its Weyl symbol is of the form
\begin{equation}\label{eq:b2-symbol-2}
    b_2(\sigma_1-\sigma)(1+r''),
\end{equation}
for some \(r''\in\Gamma^{-2\rho}_{\rho}(\R^{2d})\). Using \eqref{eq:b2-symbol} in \eqref{eq:b2-symbol-2} along with
\(\sigma^{-1}(\sigma_1-\sigma)\in\Gamma^{-2\rho}_{\rho}(\R^{2d})\), we obtain
\[
b_2(\sigma_1-\sigma)(1+r'')=\sigma^{-1}(\sigma_1-\sigma)(1+\tilde{r})(1+r'') 
\in \Gamma^{-2\rho}_{\rho}(\R^{2d})
\]
and hence
\[
B_2(T_1-T_2)\in\Psi^{-2\rho}_{\rho}(\R^d).
\]
This establishes the required property for the operator
\(K = B_2(T_1 - T_2)\), 
and thus completes the proof of
the invariance under change of quantization stated in \eqref{eq:equiv-invariance-quant}.
\textbf{Invertibility formula.}
We now prove the second statement of Theorem~\ref{thm:main-result}, namely
\[
H_{T_{\sigma^{-1}}^{-1}}(\varepsilon)\sim H_\sigma(\varepsilon),
\qquad \varepsilon\to0.
\]
To this end, we establish a parametrix relation between $T_\sigma$ and
$T_{\sigma^{-1}}$, showing that their composition differs from the identity by
a compact operator of strictly negative order.
Using the assumed invertibility of $T_{\sigma^{-1}}$, this relation allows us
to compare the entropy of $T_\sigma$ with that of $T_{\sigma^{-1}}^{-1}$.
Applying the Weyl composition formula
\cite[Theorem~23.6]{shubin_pseudodifferential_2001}
to $\sigma^{-1}$ and $\sigma$ and following a similar line of arguments as the one leading to \eqref{eq:b2-symbol}, shows that
\begin{equation}\label{eq:prod-operators-identity-comp}
    T_{\sigma^{-1}}T_\sigma
    =
    I+R,
\end{equation}
with \(R\in\Psi^{-2\rho}_\rho(\R^{d})\).

Since \(T_{\sigma^{-1}}\) is invertible by assumption, \eqref{eq:prod-operators-identity-comp}
can be rewritten as
\begin{equation}\label{eq:introduction-of-R2}
    T_\sigma
    =
    T_{\sigma^{-1}}^{-1}(I+R).
\end{equation}
By \cite[Theorem~24.4]{shubin_pseudodifferential_2001}, the operator
\(R\in\Psi^{-2\rho}_\rho(\R^{d})\) extends to a compact operator on \(L^2(\R^d)\).
Identity \eqref{eq:introduction-of-R2} is therefore of the form required in
Lemma~\ref{lem:technical-lemma}, with
\(T_1=T_\sigma\), \(T_2=T_{\sigma^{-1}}^{-1}\),
\(K=R\), and \(K_{-\infty}=0\).
Applying the lemma yields
\[
    H_{T_{\sigma^{-1}}^{-1}}(\varepsilon)
    \sim
    H_\sigma(\varepsilon),
    \qquad \varepsilon\to0,
\]
which is precisely \eqref{eq:main-result-part2}.

\begin{proof}[Proof of Lemma~\ref{lem:technical-lemma}]
Let $k$ denote the Weyl symbol of $K$.
Since $K\in\Psi^{-m}_{\rho_1}(\R^d)$ for some $m>0$, we have
$k\in\Gamma^{-m}_{\rho_1}(\R^{2d})$, hence $k$ is of strictly negative order.
The constant symbol $1$ belongs to $\mathrm{H}\Gamma^{0,0}_{\rho_1}(\R^{2d})$.
Applying Shubin’s stability result for hypoelliptic symbols under lower-order
additive perturbations \cite[Lemma~25.1 (1c)]{shubin_pseudodifferential_2001} 
yields 
\[
1+k \in \mathrm{H}\Gamma^{0,0}_{\rho_1}(\R^{2d}).
\]
Consequently, by \cite[Theorem~25.1]{shubin_pseudodifferential_2001},
the operator $I+K$ admits a right parametrix, that is, there exists a
pseudodifferential operator $P\in\Psi^0_{\rho_1}(\R^d)$ such that
\begin{equation}\label{eq:parametrix-formula}
    (I+K)P = I + R_{-\infty},
\end{equation}
for some residual operator $R_{-\infty}\in\Psi^{-\infty}(\R^d)$.
By the Weyl composition formula \cite[Theorem~23.6]{shubin_pseudodifferential_2001},
the Weyl symbol of $(I+K)P$ has the form
\[
(1+k)p + q,
\]
where $p$ is the Weyl symbol of $P$ and
$q\in\Gamma^{-2\rho_1}_{\rho_1}(\R^{2d})$.
On the other hand, by \eqref{eq:parametrix-formula}, the Weyl symbol of $(I+K)P$
is $1+r_{-\infty}$, where $r_{-\infty}$ is the symbol of $R_{-\infty}$.
Using that $k\in\Gamma^{-m}_{\rho_1}(\R^{2d})$ and comparing both expressions,
it follows that there exists a symbol
\[
r_1 \in \Gamma^{-\gamma}_{\rho_1}(\R^{2d}),
\qquad
\gamma \coloneqq \min\{2\rho_1,m\},
\]
such that
\[
p = 1 + r_1.
\]
Equivalently, in operator form,
\[
P = I + R_1,
\qquad
R_1 \in \Psi^{-\gamma}_{\rho_1}(\R^d).
\]
We now multiply the identity
\[
T_1 = T_2 (I+K) + K_{-\infty}
\]
from the right by $P$ and use \eqref{eq:parametrix-formula} to obtain
\begin{align*}
T_1 P
&= T_2 (I+K)P + K_{-\infty}P \\
&= T_2 (I+R_{-\infty}) + K_{-\infty}P \\
&= T_2 + T_2 R_{-\infty} + K_{-\infty}P.
\end{align*}
Rearranging terms yields
\begin{equation}\label{eq:introduction-of-R1}
    T_2
    =
    T_1 (I+R_1) + R'_{-\infty},
\end{equation}
where
\[
R'_{-\infty} \coloneqq -T_2 R_{-\infty} - K_{-\infty}P
\in \Psi^{-\infty}(\R^d),
\]
since $\Psi^{-\infty}(\R^d)$ is a two-sided ideal in the pseudodifferential
calculus. Equation \eqref{eq:introduction-of-R1} provides a representation of $T_2$
of the same structural form as the hypothesis of
Lemma~\ref{lem:technical-lemma}, namely
\[
T_2 = T_1 (I+R_1) + R'_{-\infty},
\]
with $R_1\in\Psi^{-\gamma}_{\rho_1}(\R^d)$ for some $\gamma>0$ and
$R'_{-\infty}\in\Psi^{-\infty}(\R^d)$.
Although the order $\gamma$ may differ from the original exponent $m$,
the lemma only requires the perturbation to be of strictly negative order.
This symmetry explains why Lemma~\ref{lem:technical-lemma} only requires that
either $T_1$ or $T_2$ be the Weyl quantization of~$\sigma$.

We now turn to the comparison of the entropy numbers of $T_1$ and $T_2$.
Recall that the entropy numbers $\{e_n(T)\}_{n\in\N}$ of a compact operator
$T\colon L^2(\R^d)\to L^2(\R^d)$ are defined by
\begin{equation}\label{eq:def-entropy-nbs}
    e_n(T)
    \coloneqq
    \inf\bigl\{ \varepsilon>0 \,\big|\, H_T(\varepsilon)\le n \bigr\},
    \qquad n\in\N.
\end{equation}
We refer to
\cite{opideals1980pietsch,bookeigen1987pietsch,carlEntropyCompactnessApproximation1990,lorentzConstructiveApproximationAdvanced1996,edmundsFunctionSpacesEntropy1996}
for background on entropy numbers and their basic properties. Starting from the representation \eqref{eq:introduction-of-R1}, and using the
submultiplicativity and subadditivity of entropy numbers 
(see, for example,
\cite[Section~15.7]{lorentzConstructiveApproximationAdvanced1996} or
\cite[Section~1.3.1, Lemma~1%
]{edmundsFunctionSpacesEntropy1996}),
we obtain
\begin{equation}\label{eq:entropy-nb-arg-1}
    e_{\,n+5\lfloor n^{1/2}\rfloor}(T_2)
    \le
    e_n(T_1)\,e_{\lfloor n^{1/2}\rfloor}(I+R_1)
    +
    e_{4\lfloor n^{1/2}\rfloor}(R'_{-\infty}),
\end{equation}
together with
\begin{equation}\label{eq:entropy-nb-arg-2}
    e_{\lfloor n^{1/2}\rfloor}(I+R_1)
    \le
    e_{0}(I) + e_{\lfloor n^{1/2}\rfloor}(R_1),
    \qquad n\in\N.
\end{equation}
Since $e_{0}(I)=\|I\|_{2\to2}=1$ 
\cite[Section~15.7]{lorentzConstructiveApproximationAdvanced1996},
and $R_1$ by virtue of being compact satisfies
\[
\lim_{{j}\to\infty} e_{j}(R_1)=0,
\]
it follows, 
by \eqref{eq:entropy-nb-arg-2},
that
\begin{equation}\label{eq:bound-entropy-identity}
    e_{\lfloor n^{1/2}\rfloor}(I+R_1)
    \le
    1+o(1),
    \qquad n\to\infty.
\end{equation}

We now estimate the term
$e_{4\lfloor n^{1/2}\rfloor}\!\left(R'_{-\infty}\right)$
appearing in \eqref{eq:entropy-nb-arg-1}.
To this end, arbitrarily fix $s>0$ and consider the hypoelliptic symbol
$\bar\sigma\in \mathrm{H}\Gamma^{-s,-s}_{\rho_1}(\R^{2d})$
defined according to $\bar\sigma(z) = (1+\|z\|_2^2)^{-s/2}$, for all $z\in\R^{2d}$.
Let $T_{\bar\sigma}$ denote the Weyl quantization of $\bar\sigma$ and let $Q\in\Psi^{s}_{\rho_1}(\R^d)$ be a right parametrix, that is, 
\begin{equation*}
    T_{\bar\sigma} Q = I + R''_{-\infty}, 
    \quad  R''_{-\infty} \in \Psi^{-\infty}(\R^d).
\end{equation*}

One verifies that 
\begin{equation}\label{eq:proof-lemma-AAA}
    V_{\bar\sigma}(\lambda)
    =  \int_{\R^{2d}} \mathbbm{1}_{\{(1+\|z\|_2^2)^{-s/2}>\lambda\}}  \,dz
    \sim \omega_{2d}\,\lambda^{-\frac{2d}{s}}, 
    \quad  \lambda \to 0.
\end{equation}
Applying the entropy--volume relation \eqref{eq:main-res-volume-integral},
already established for Weyl quantization in \eqref{eq:mainr-result-for-Weyl-1}, to the operator
$T_{\bar\sigma}$ yields
\[
H_{T_{\bar\sigma}}(\varepsilon)
\sim
\int_\varepsilon^\infty \frac{V_{\bar\sigma}(\lambda)}{\lambda}\,d\lambda
\sim \frac{s \, \omega_{2d}}{2d}\varepsilon^{-\frac{2d}{s}},
\qquad \varepsilon\to0.
\]
By the definition of the entropy numbers \eqref{eq:def-entropy-nbs}, it follows that
\begin{equation}\label{eq:proof-lemma-AAB}
    e_j(T_{\bar\sigma})
    \sim
    \left(\frac{2d\,j}{s\,\omega_{2d}}\right)^{-\frac{s}{2d}},
    \qquad j\to\infty.
\end{equation}
Evaluating \eqref{eq:proof-lemma-AAB} at \(j=\lfloor n^{1/2}\rfloor\) yields
\begin{equation}\label{eq:poly-scaling-ent-nb}
    e_{\lfloor n^{1/2}\rfloor}(T_{\bar\sigma})
    \le
    C\,n^{-\frac{s}{4d}}\bigl(1+o(1)\bigr),
    \qquad n\to\infty,
\end{equation}
for a suitable constant \(C>0\).

We now return to the operator $R'_{-\infty}$.
Using the factorization
\[
R'_{-\infty} = \left(T_{\bar\sigma}Q - R''_{-\infty}\right)R'_{-\infty},
\]
and applying the subadditivity and submultiplicativity of entropy numbers, we obtain
\begin{align}
e_{4\lfloor n^{1/2}\rfloor}\!\left(R'_{-\infty}\right)
&\le
e_{\lfloor n^{1/2}\rfloor}(T_{\bar\sigma})\,
e_{\lfloor n^{1/2}\rfloor}\!\left(Q R'_{-\infty}\right)
+
e_{\lfloor n^{1/2}\rfloor}\!\left(R''_{-\infty}\right)
e_{\lfloor n^{1/2}\rfloor}\!\left(R'_{-\infty}\right)
\label{eq:bound-ent-split-1}\\
&\le
C\,n^{-\frac{s}{4d}}\bigl(1+o(1)\bigr)\,
e_{\lfloor n^{1/2}\rfloor}\!\left(Q R'_{-\infty}\right)
+
e_{\lfloor n^{1/2}\rfloor}\!\left(R''_{-\infty}\right)
e_{\lfloor n^{1/2}\rfloor}\!\left(R'_{-\infty}\right),
\label{eq:bound-ent-split-2}
\end{align}
where \eqref{eq:bound-ent-split-2} follows from \eqref{eq:poly-scaling-ent-nb}.
Since $Q\in\Psi^{s}_{\rho_1}(\R^d)$ and $R'_{-\infty}\in\Psi^{-\infty}(\R^d)$,
their composition $Q R'_{-\infty}$ belongs to $\Psi^{-\infty}(\R^d)$ and is therefore compact on $L^2(\R^d)$.
Moreover, $R''_{-\infty}\in\Psi^{-\infty}(\R^d)$ is also compact.
Consequently,
\[
\lim_{n\to\infty}
e_{\lfloor n^{1/2}\rfloor}\!\left(Q R'_{-\infty}\right)
=
0,
\qquad
\lim_{n\to\infty}
e_{\lfloor n^{1/2}\rfloor}\!\left(R''_{-\infty}\right)
=
0.
\]
Combining these limits with \eqref{eq:bound-ent-split-1}–\eqref{eq:bound-ent-split-2} yields
\begin{equation}\label{eq:R-fast-decay}
e_{4\lfloor n^{1/2}\rfloor}\!\left(R'_{-\infty}\right)
=
o \left(n^{-\frac{s}{4d}} + e_{\lfloor n^{1/2}\rfloor}\!\left(R'_{-\infty}\right)\right),
\qquad n\to\infty.
\end{equation}
Iterating \eqref{eq:R-fast-decay} and using that $s>0$ is arbitrary yields that, for every $q>0$,
\begin{equation}\label{eq:R-fast-decay-q}
e_{4\lfloor n^{1/2}\rfloor} \left(R'_{-\infty}\right)
=
o \left(n^{-q}\right),
\qquad n\to\infty.
\end{equation}
From now on, we assume without loss of generality that \(T_1=T_\sigma\), where
\(T_\sigma\) is the Weyl quantization of the hypoelliptic symbol \(\sigma\).
If instead \(T_2=T_\sigma\), the same argument applies after interchanging the
roles of \(T_1\) and \(T_2\) from the outset.

An argument identical to that used to derive \eqref{eq:R-fast-decay-q} shows
that the entropy numbers of the 
operator \(K_{-\infty}\) in the
parametrix relation
\begin{equation}\label{eq:parametrix-T1-T2}
T_1 = T_2(I+K)+K_{-\infty}
\end{equation}
also decay faster than any polynomial rate. More precisely, for every \(q>0\),
\begin{equation}\label{eq:K-fast-decay-q}
    e_{4\lfloor n^{1/2}\rfloor}\!\left(K_{-\infty}\right)
    =
    o \left(n^{-q}\right),
    \qquad n\to\infty.
\end{equation}
We next derive two-sided polynomial bounds for the entropy numbers of
\(T_1=T_\sigma\).
Since $\sigma \in \mathrm{H}\Gamma^{-}_{\rho}(\R^{2d})$, there exist exponents
$m_+ \ge m_- > 0$ such that the two-sided hypoellipticity bounds
\eqref{eq:bound-def-hypoellip} hold. Consequently, there exist constants
$c_1,c_2>0$ and $\lambda_0>0$ so that, for all $\lambda\in(0,\lambda_0)$,
\[
c_1 \lambda^{-2d/m_+}
\le
V_\sigma(\lambda)
\le
c_2 \lambda^{-2d/m_-}.
\]
Since \eqref{eq:main-res-volume-integral} has already been established for the
Weyl quantization, this yields
\[
H_{T_1}(\varepsilon)
\sim
\int_\varepsilon^\infty \frac{V_\sigma(\lambda)}{\lambda}\,d\lambda,
\qquad \varepsilon\to0,
\]
and therefore
\[
c_3 \, \varepsilon^{-2d/m_+}
\le
H_{T_1}(\varepsilon)
\le
c_4 \, \varepsilon^{-2d/m_-},
\qquad \varepsilon\in(0,\varepsilon_0),
\]
for suitable constants \(c_3,c_4,\varepsilon_0>0\).
By the definition of the entropy numbers, these bounds imply the existence of
constants \(c_5,c_6>0\) such that
\begin{equation}\label{eq:two-sided-polynomial-T1}
    c_5\, n^{-m_+/(2d)}
\lesssim
    e_n(T_1)
\lesssim
    c_6\, n^{-m_-/(2d)},
    \qquad n\to\infty.
\end{equation}
Choose \(q>m_+/(2d)\). Then \eqref{eq:R-fast-decay-q} and
\eqref{eq:two-sided-polynomial-T1} imply
\[
e_{4\lfloor n^{1/2}\rfloor}\!\left(R'_{-\infty}\right)
=
o \bigl(e_n(T_1)\bigr),
\qquad n\to\infty.
\]
Combining this with \eqref{eq:entropy-nb-arg-1} and
\eqref{eq:bound-entropy-identity}, we obtain
\begin{equation}\label{eq:inequality-entropy-nbs}
    e_{\,n+5\lfloor n^{1/2}\rfloor}(T_2)
    \le
    e_n(T_1)\bigl(1+o(1)\bigr),
    \qquad n\to\infty.
\end{equation}
Based on \eqref{eq:parametrix-T1-T2}, applying the same entropy-number calculus as before, we get
\begin{equation}\label{eq:entropy-nb-arg-1-bis}
    e_{\,n+5\lfloor n^{1/2}\rfloor}(T_1)
    \le
    e_n(T_2)\,e_{\lfloor n^{1/2}\rfloor}(I+K)
    +
    e_{4\lfloor n^{1/2}\rfloor}(K_{-\infty}).
\end{equation}
Since \(K\) is compact,
\[
e_{\lfloor n^{1/2}\rfloor}(I+K)\le 1+o(1),
\qquad n\to\infty.
\]
Moreover, by \eqref{eq:K-fast-decay-q},
\[
e_{4\lfloor n^{1/2}\rfloor}(K_{-\infty})
=
o \left(n^{-q}\right),
\]
for every $q>0$.
Using the lower bound in \eqref{eq:two-sided-polynomial-T1} together with
\eqref{eq:entropy-nb-arg-1-bis}, we infer that
\[
e_n(T_2)\gtrsim n^{-m_+/(2d)},
\qquad n\to\infty.
\]
In particular,
\[
e_{4\lfloor n^{1/2}\rfloor}(K_{-\infty})
=
o \bigl(e_n(T_2)\bigr),
\qquad n\to\infty.
\]
Reinserting this into \eqref{eq:entropy-nb-arg-1-bis}, we obtain
\begin{equation}\label{eq:inequality-entropy-nbs-bis}
    e_{\,n+5\lfloor n^{1/2}\rfloor}(T_1)
    \le
    e_n(T_2)\bigl(1+o(1)\bigr),
    \qquad n\to\infty.
\end{equation}
The preceding argument was carried out under the assumption \(T_1=T_\sigma\).
Since the parametrix relation assumed in Lemma~\ref{lem:technical-lemma} and the
derived reverse parametrix relation are invariant under exchanging \(T_1\) and
\(T_2\), the same reasoning applies when \(T_2=T_\sigma\).
Consequently, the estimates
\eqref{eq:inequality-entropy-nbs} and
\eqref{eq:inequality-entropy-nbs-bis}
hold irrespective of which of the two operators is the Weyl quantization of
\(\sigma\).

\color{black}

We now convert the entropy-number relations
\eqref{eq:inequality-entropy-nbs} and \eqref{eq:inequality-entropy-nbs-bis}
into corresponding asymptotic relations for the entropy functions.
Fix $\varepsilon>0$ and choose $n_\varepsilon\in\N$ such that
\begin{equation}\label{eq:definition-k-entropy}
    n_\varepsilon+5\lfloor n_\varepsilon^{1/2}\rfloor
    <
    H_{T_2}(\varepsilon)
    \le
  n_\varepsilon+1+5\lfloor (n_\varepsilon+1)^{1/2}\rfloor.
\end{equation}
By the definition of entropy numbers this implies
\[
e_{\,n_\varepsilon+5\lfloor n_\varepsilon^{1/2}\rfloor}(T_2)\ge \varepsilon .
\]
Combining this with \eqref{eq:inequality-entropy-nbs}, we obtain
\[
\varepsilon
\le
e_{\,n_\varepsilon+5\lfloor n_\varepsilon^{1/2}\rfloor}(T_2)
\le
e_{n_\varepsilon}(T_1)\bigl(1+o(1)\bigr),
\qquad \varepsilon\to0 .
\]
Hence there exists a function $\zeta:\R_+^*\to\R_+^*$ with
$\zeta(\varepsilon)\to1$ as $\varepsilon\to0$ such that
\begin{equation}\label{eq:relation-eps-entropy-nbs}
    \varepsilon\,\zeta(\varepsilon)
    <
    e_{n_\varepsilon}(T_1).
\end{equation}
Applying again the definition of entropy numbers yields
\begin{equation}\label{eq:useful-relation1}
    H_{T_1}\!\left(\varepsilon\,\zeta(\varepsilon)\right)
    \ge
    n_\varepsilon
    =
    H_{T_2}(\varepsilon)\bigl(1+o_{\varepsilon\to0}(1)\bigr).
\end{equation}
To obtain an upper bound for \(H_{T_1}\) in terms of \(H_{T_2}\), let $m_\varepsilon := \left\lceil H_{T_2}(\varepsilon) \right\rceil$.
Then $m_\varepsilon\to\infty$ as $\varepsilon\to0$, and by the definition of
entropy numbers we have
\[
e_{m_\varepsilon}(T_2)\le \varepsilon .
\]
Applying \eqref{eq:inequality-entropy-nbs-bis} with $n=m_\varepsilon$ yields
\[
e_{\,m_\varepsilon+5\lfloor m_\varepsilon^{1/2}\rfloor}(T_1)
\le
\varepsilon\bigl(1+o(1)\bigr),
\qquad \varepsilon\to0 .
\]
Consequently, there exists a function
$\zeta':\mathbb{R}_+^*\to\mathbb{R}_+^*$ with
$\zeta'(\varepsilon)\to1$ as $\varepsilon\to0$ such that
\begin{equation}\label{eq:useful-relation1-bis}
    H_{T_1}\!\left(\varepsilon\,\zeta'(\varepsilon)\right)
    \le
    m_\varepsilon+5\lfloor m_\varepsilon^{1/2}\rfloor.
\end{equation}
Since
\[
m_\varepsilon
=
H_{T_2}(\varepsilon)+O(1),
\qquad \varepsilon\to0,
\]
and $H_{T_2}(\varepsilon)\to\infty$, we obtain
\[
m_\varepsilon+5\lfloor m_\varepsilon^{1/2}\rfloor
=
H_{T_2}(\varepsilon)\bigl(1+o_{\varepsilon\to0}(1)\bigr).
\]
Substituting this relation into \eqref{eq:useful-relation1-bis} yields
\begin{equation}\label{eq:asymp-upper}
    H_{T_1} \left(\varepsilon\,\zeta'(\varepsilon)\right)
\le
H_{T_2}(\varepsilon)\bigl(1+o_{\varepsilon\to0}(1)\bigr).
\end{equation}

From here on we assume, without loss of generality, that
$T_1=T_\sigma$. The case $T_2=T_\sigma$ is treated in the
same way after interchanging the roles of $T_1$ and $T_2$. Since \eqref{eq:main-res-volume-integral} has already been established
for the Weyl quantization, we obtain
\begin{align}
H_{T_1}\!\left(\varepsilon\,\zeta(\varepsilon)\right)
&\sim
\int_{\varepsilon\zeta(\varepsilon)}^\infty
\frac{V_\sigma(\lambda)}{\lambda}\,d\lambda
\label{eq:split-integral-zeta-proof-A1}\\
&=
\int_{\varepsilon}^{\infty}\frac{V_\sigma(\lambda)}{\lambda}\,d\lambda
+
\int_{\varepsilon\zeta(\varepsilon)}^{\varepsilon}
\frac{V_\sigma(\lambda)}{\lambda}\,d\lambda
\notag\\
&\sim
H_{T_1}(\varepsilon)
+
\int_{\varepsilon\zeta(\varepsilon)}^{\varepsilon}
\frac{V_\sigma(\lambda)}{\lambda}\,d\lambda,
\label{eq:split-integral-zeta-proof-B}
\end{align}
as $\varepsilon\to0$, where the second integral is understood with its
natural orientation. Using the monotonicity of $V_\sigma$, the remainder term in
\eqref{eq:split-integral-zeta-proof-A1}--\eqref{eq:split-integral-zeta-proof-B}
satisfies
\begin{equation}\label{eq:bound-remainder-volume-integral-A}
\left|
\int_{\varepsilon\zeta(\varepsilon)}^{\varepsilon}
\frac{V_\sigma(\lambda)}{\lambda}\,d\lambda
\right|
\le
\frac{V_\sigma \left(\varepsilon\,\zeta^*(\varepsilon)\right)}
     {\varepsilon\,\zeta^*(\varepsilon)}
\,
\bigl|\varepsilon-\varepsilon\,\zeta(\varepsilon)\bigr|
=
\frac{V_\sigma \left(\varepsilon\,\zeta^*(\varepsilon)\right)}
     {\zeta^*(\varepsilon)}
\,
|1-\zeta(\varepsilon)|,
\end{equation}
for all $\varepsilon>0$, where
$\zeta^*(\varepsilon)=\min\{1,\zeta(\varepsilon)\}$.
Since $V_\sigma$ is regularly varying at zero, Karamata’s theorem implies
\begin{equation}\label{eq:volume-integral-compp-A}
V_\sigma \left(\varepsilon\,\zeta^*(\varepsilon)\right)
=
O_{\varepsilon\to0} \left(
\int_{\varepsilon\zeta^*(\varepsilon)}^\infty
\frac{V_\sigma(\lambda)}{\lambda}\,d\lambda
\right).
\end{equation}
Because $\zeta(\varepsilon)\to1$ as $\varepsilon\to0$, it follows from
\eqref{eq:volume-integral-compp-A} that
\begin{equation}\label{eq:growth-control-volume-A}
V_\sigma \left(\varepsilon\,\zeta^*(\varepsilon)\right)
|1-\zeta(\varepsilon)|
=
o_{\varepsilon\to0}\!\left(
\int_{\varepsilon\zeta^*(\varepsilon)}^\infty
\frac{V_\sigma(\lambda)}{\lambda}\,d\lambda
\right).
\end{equation}
Combining
\textcolor{purple}{\eqref{eq:split-integral-zeta-proof-A1}--\eqref{eq:split-integral-zeta-proof-B}},
\eqref{eq:bound-remainder-volume-integral-A},
and \eqref{eq:growth-control-volume-A} therefore yields
\begin{equation}\label{eq:useful-relation3-A}
H_{T_1}\!\left(\varepsilon\,\zeta(\varepsilon)\right)
\sim
H_{T_1}(\varepsilon),
\qquad \varepsilon\to0 .
\end{equation}
The same argument, with $\zeta$ replaced by $\zeta'$, gives
\begin{equation}\label{eq:useful-relation3-bis-A}
H_{T_1}\!\left(\varepsilon\,\zeta'(\varepsilon)\right)
\sim
H_{T_1}(\varepsilon),
\qquad \varepsilon\to0 .
\end{equation}
Finally, combining
\eqref{eq:useful-relation1},
\eqref{eq:asymp-upper},
\eqref{eq:useful-relation3-A},
and \eqref{eq:useful-relation3-bis-A},
we conclude that
\[
H_{T_1}(\varepsilon)
\sim
H_{T_2}(\varepsilon),
\qquad \varepsilon\to0,
\]
which completes the proof of Lemma~\ref{lem:technical-lemma}.
\end{proof}

\color{black}

\subsection{Proof of Corollary~\ref{cor:ME-Main}}\label{sec:proof-ME-Main}

From the definition \eqref{eq:definition-volume-function} of $V_{\symbolnot}$, we have
\begin{align*}
\int_{\varepsilon}^{\infty}\frac{V_{\symbolnot}(\lambda)}{\lambda}\,d\lambda
&=
\int_{\varepsilon}^{\infty}
\int_{\mathbb{R}^{2d}}
\frac{\mathbbm{1}_{\{\symbolnot(x,\omega)>\lambda\}}}{\lambda}
\,dx\,d\omega\,d\lambda \\
&=
\int_{0}^{\infty}
\int_{\mathbb{R}^{2d}}
\frac{\mathbbm{1}_{\{\symbolnot(x,\omega)>\lambda>\varepsilon\}}}{\lambda}
\,dx\,d\omega\,d\lambda ,
\qquad \varepsilon>0.
\end{align*}
Since the integrand is nonnegative, Tonelli’s 
theorem yields
\[
\int_{0}^{\infty}
\int_{\mathbb{R}^{2d}}
\frac{\mathbbm{1}_{\{\symbolnot(x,\omega)>\lambda>\varepsilon\}}}{\lambda}
\,dx\,d\omega\,d\lambda
=
\int_{\mathbb{R}^{2d}}
\int_{0}^{\infty}
\frac{\mathbbm{1}_{\{\symbolnot(x,\omega)>\lambda>\varepsilon\}}}{\lambda}
\,d\lambda\,dx\,d\omega, \qquad \varepsilon>0.
\]
For fixed $(x,\omega)\in\mathbb{R}^{2d}$, the inner integral can be computed explicitly:
\[
\int_{0}^{\infty}
\frac{\mathbbm{1}_{\{\symbolnot(x,\omega)>\lambda>\varepsilon\}}}{\lambda}
\,d\lambda
=
\ln_{+}\!\left(\frac{\symbolnot(x,\omega)}{\varepsilon}\right),
\]
where $\ln_{+}(t)=\max\{\ln (t),0\}$.
Consequently,
\[
\int_{\varepsilon}^{\infty}\frac{V_{\symbolnot}(\lambda)}{\lambda}\,d\lambda
=
\int_{\mathbb{R}^{2d}}
\ln_{+}\!\left(\frac{\symbolnot(x,\omega)}{\varepsilon}\right)
\,dx\,d\omega ,
\qquad \varepsilon>0.
\]
The claim now follows directly from Theorem~\ref{thm:main-result}.

\subsection{Proof of Corollary~\ref{cor:MR-Main}}\label{sec:proof-MR-Main}

We combine the theory of regular variation
(see, e.g., \cite{binghamRegularVariation1987})
with Theorem~\ref{thm:main-result} to derive the asymptotics of the
type-$\tau$ integrals \eqref{eq:type-tau-integrals}.
Specifically, for $\tau=2$ and $\tau=3$, definition \eqref{eq:type-tau-integrals} gives
\[
I_2(\varepsilon)
=
\int_{\varepsilon}^{\infty}
M_{\symbolnot}(\lambda)\,\lambda^{-2}\,d\lambda,
\qquad
I_3(\varepsilon)
=
\int_{\varepsilon}^{\infty}
M_{\symbolnot}(\lambda)\,\lambda^{-3}\,d\lambda.
\]
By Theorem~\ref{thm:main-result}, we have
\[
H_{\symbolnot}(\varepsilon)
\sim
\int_{\varepsilon}^{\infty}
\frac{V_{\symbolnot}(\lambda)}{\lambda}\,d\lambda,
\qquad \varepsilon\to 0,
\]
for all quantizations. Since $V_{\symbolnot}$ is regularly varying at zero,
condition~(RC) holds and \cite[Theorem~2]{allard2025metricentropyminimaxrisk} yields
\[
H_{\symbolnot}(\varepsilon)\sim I_1(\varepsilon),
\]
which, in turn, implies
\[
I_1(\varepsilon)
\sim
\int_{\varepsilon}^{\infty}
\frac{V_{\symbolnot}(\lambda)}{\lambda}\,d\lambda,
\qquad \varepsilon\to 0.
\]
Using integration by parts, we obtain
\[
I_2(\varepsilon)
=
\varepsilon^{-1}I_1(\varepsilon)
-
\int_{\varepsilon}^{\infty}
\frac{I_1(\lambda)}{\lambda^2}\,d\lambda,
\]
and
\[
I_3(\varepsilon)
=
\varepsilon^{-2}I_1(\varepsilon)
-
2\int_{\varepsilon}^{\infty}
\frac{I_1(\lambda)}{\lambda^3}\,d\lambda.
\]
Since $I_1$ is regularly varying at zero, another application of
Karamata’s theorem implies
\[
I_2(\varepsilon)
\sim
\int_{\varepsilon}^{\infty}
\frac{V_{\symbolnot}(\lambda)}{\lambda^{2}}\,d\lambda,
\qquad
I_3(\varepsilon)
\sim
\int_{\varepsilon}^{\infty}
\frac{V_{\symbolnot}(\lambda)}{\lambda^{3}}\,d\lambda,
\qquad \varepsilon\to 0.
\]
In particular, by \cite[Theorem~4]{allard2025metricentropyminimaxrisk},
\begin{equation}\label{eq:proof-risk-1}
R_{\symbolnot}(\kappa)
\sim
\kappa^2\,\varepsilon_\kappa
\int_{\varepsilon_\kappa}^{\infty}\frac{V_{\symbolnot}(\lambda)}{\lambda^2}\,d\lambda,
\qquad \kappa\to0,
\end{equation}
where the critical radius $\varepsilon_\kappa$ is determined by
\begin{equation}\label{eq:proof-risk-4}
\kappa^2
\int_{\varepsilon_\kappa}^{\infty}
V_{\symbolnot}(\lambda)
\left(\frac{2}{\lambda^3}-\frac{1}{\lambda^2\varepsilon_\kappa}\right)
\,d\lambda
=1,
\qquad \kappa>0.
\end{equation}
We next compute the integrals appearing in \eqref{eq:proof-risk-1} and
\eqref{eq:proof-risk-4} explicitly.
From the definition \eqref{eq:definition-volume-function} of $V_{\symbolnot}$,
\begin{align*}
\int_{\varepsilon}^{\infty}\frac{V_{\symbolnot}(\lambda)}{\lambda^2}\,d\lambda
&=
\int_{\varepsilon}^{\infty}
\int_{\mathbb{R}^{2d}}
\frac{\mathbbm{1}_{\{\symbolnot(x,\omega)>\lambda\}}}{\lambda^2}
\,dx\,d\omega\,d\lambda \\
&=
\int_{0}^{\infty}
\int_{\mathbb{R}^{2d}}
\frac{\mathbbm{1}_{\{\symbolnot(x,\omega)>\lambda>\varepsilon\}}}{\lambda^2}
\,dx\,d\omega\,d\lambda .
\end{align*}
Since the integrand is nonnegative, Tonelli’s theorem allows us to exchange the
order of integration, giving
\[
\int_{\mathbb{R}^{2d}}
\int_{0}^{\infty}
\frac{\mathbbm{1}_{\{\symbolnot(x,\omega)>\lambda>\varepsilon\}}}{\lambda^2}
\,d\lambda\,dx\,d\omega .
\]
For fixed $(x,\omega)\in\mathbb{R}^{2d}$, the inner integral can be evaluated explicitly:
\[
\int_{0}^{\infty}
\frac{\mathbbm{1}_{\{\symbolnot(x,\omega)>\lambda>\varepsilon\}}}{\lambda^2}
\,d\lambda
=
\left(\frac{1}{\varepsilon}-\frac{1}{\symbolnot(x,\omega)}\right)_+ .
\]
Consequently,
\begin{equation}\label{eq:proof-risk-2}
\int_{\varepsilon}^{\infty}\frac{V_{\symbolnot}(\lambda)}{\lambda^2}\,d\lambda
=
\int_{\mathbb{R}^{2d}}
\left(\frac{1}{\varepsilon}-\frac{1}{\symbolnot(x,\omega)}\right)_+
\,dx\,d\omega .
\end{equation}
Substituting \eqref{eq:proof-risk-2} into \eqref{eq:proof-risk-1} yields
\eqref{eq:main-res-risk-1-corr}.

A completely analogous computation gives
\begin{equation}\label{eq:proof-risk-3}
\int_{\varepsilon}^{\infty}\frac{V_{\symbolnot}(\lambda)}{\lambda^3}\,d\lambda
=
\int_{\mathbb{R}^{2d}}
\left(\frac{1}{2\varepsilon^2}-\frac{1}{2\symbolnot(x,\omega)^2}\right)_+
\,dx\,d\omega .
\end{equation}
Combining \eqref{eq:proof-risk-2} and \eqref{eq:proof-risk-3}, we obtain
\begin{align*}
\int_{\varepsilon}^{\infty}
V_{\symbolnot}(\lambda)
\left(\frac{2}{\lambda^3}-\frac{1}{\lambda^2\varepsilon}\right)
\,d\lambda
&=
\int_{\mathbb{R}^{2d}}
\left[
\left(\frac{1}{\varepsilon^2}-\frac{1}{\symbolnot(x,\omega)^2}\right)_+
-
\left(\frac{1}{\varepsilon^2}-\frac{1}{\symbolnot(x,\omega)\varepsilon}\right)_+
\right]
dx\,d\omega \\
&=
\int_{\mathbb{R}^{2d}}
\frac{1}{\symbolnot(x,\omega)}
\left(\frac{1}{\varepsilon}-\frac{1}{\symbolnot(x,\omega)}\right)_+
\,dx\,d\omega .
\end{align*}
Inserting this expression into \eqref{eq:proof-risk-4} yields
\eqref{eq:main-res-risk-2-corr}, thereby completing the proof.

\subsection{Proof of Theorem~\ref{thm:sob-sch}}\label{sec:proof-thm-sob-sch}

By \cite[Proposition~25.4]{shubin_pseudodifferential_2001}, $\ballsch$ is compact,
and therefore bounded, in $L^2(\mathbb R^d)$.
We next show that $\ker \tsch=\{0\}$.
If $\tsch\psi=0$ for some nonzero $\psi\in L^2(\R^d)$, then
$\mu\psi\in\ballsch$ for all $\mu\in\R$, which contradicts boundedness of
$\ballsch$ in $L^2(\R^d)$.
Since $\tsch$ is self-adjoint, $\ker \tsch^*=\ker \tsch=\{0\}$, and
\cite[Theorem~25.4]{shubin_pseudodifferential_2001} therefore guarantees the
existence of an inverse $\tsch^{-1}$,
which can be extended via \cite[Theorem~24.4]{shubin_pseudodifferential_2001} to a compact operator 
on 
$L^2(\mathbb R^d)$.
Setting $g=\tsch f$ in \eqref{eq:definition-ballsch}, we obtain
\[
\ballsch
=
\{\tsch^{-1}g:\ \|g\|_{L^2(\mathbb R^d)}\le1\}
=
\tsch^{-1}\mathcal B_2,
\]
where $\mathcal B_2$ denotes the unit ball in $L^2(\mathbb R^d)$.
Consequently,
\begin{equation}\label{eq:equality-metric-entropy-op-entropy-sch}
H\!\left(\varepsilon;\ballsch,\|\cdot\|_{L^2(\mathbb R^d)}\right)
=
H_{\tsch^{-1}}(\varepsilon),
\qquad \varepsilon>0.
\end{equation}

To apply Theorem~\ref{thm:main-result}, we study the volume function
\begin{equation}\label{eq:integral-volume-sch-sob1}
V_{\symbolnotsch^{-1}}(\lambda)
=
\int_{\mathbb R^{2d}}
\mathbbm{1}_{\{(1+c\|x\|_2^2+(2\pi\|\omega\|_2)^2)^{-s/2}>\lambda\}}
\,dx\,d\omega,
\qquad \lambda>0.
\end{equation}
Performing the change of variables $x\mapsto x'/\sqrt c$ and
$\omega\mapsto\omega'/(2\pi)$ yields
\begin{align}
V_{\symbolnotsch^{-1}}(\lambda)
&=
\frac{1}{(2\pi\sqrt c)^d}
\int_{\mathbb R^{2d}}
\mathbbm{1}_{\{(1+\|x'\|_2^2+\|\omega'\|_2^2)^{-s/2}>\lambda\}}
\,dx'\,d\omega' \label{eq:volume-integral-sch-formula1}\\
&=
\frac{1}{(2\pi\sqrt c)^d}
\int_{\mathbb R^{2d}}
\mathbbm{1}_{\{\|x'\|_2^2+\|\omega'\|_2^2\, <\, \lambda^{-2/s}-1\}}
\,dx'\,d\omega'. \label{eq:volume-integral-sch-formula2}
\end{align}
Hence $V_{\symbolnotsch^{-1}}(\lambda)$ is proportional to the volume of a ball of
radius $(\lambda^{-2/s}-1)^{1/2}$ in $\mathbb R^{2d}$, and therefore
\[
V_{\symbolnotsch^{-1}}(\lambda)
=
\frac{\omega_{2d}}{(2\pi\sqrt c)^d}(\lambda^{-2/s}-1)^d
\sim
\frac{\omega_{2d}}{(2\pi\sqrt c)^d}\lambda^{-2d/s},
\qquad \lambda\to0.
\]
A direct calculation further shows that
\begin{equation}\label{eq:integral-sob-vol-1}
\int_{\varepsilon}^{\infty}
\frac{V_{\symbolnotsch^{-1}}(\lambda)}{\lambda}\,d\lambda
\textcolor{blue}{\sim}
\frac{s\,\omega_{2d}}{2d(2\pi\sqrt c)^d}\,\varepsilon^{-2d/s},
\qquad \varepsilon \textcolor{blue}{\to} 0.
\end{equation}
The symbol $\symbolnotsch^{-1}$ satisfies the assumptions of
Theorem~\ref{thm:main-result}.
Applying successively
\eqref{eq:equality-metric-entropy-op-entropy-sch},
\eqref{eq:main-result-part2},
\eqref{eq:main-res-volume-integral},
and \eqref{eq:integral-sob-vol-1}, we obtain
\begin{align*}
H\!\left(\varepsilon;\ballsch,\|\cdot\|_{L^2(\mathbb R^d)}\right)
&=
H_{\tsch^{-1}}(\varepsilon)
\sim
H_{\symbolnotsch^{-1}}(\varepsilon)\\
&\sim
\int_{\varepsilon}^{\infty}
\frac{V_{\symbolnotsch^{-1}}(\lambda)}{\lambda}\,d\lambda
\sim
\frac{s\,\omega_{2d}}{2d(2\pi\sqrt c)^d}\,\varepsilon^{-2d/s},
\qquad \varepsilon\to0,
\end{align*}
which establishes \eqref{eq:pinsker-entropy}.

Finally, by \cite[Appendix~A, Table~(iv)]{allard2025metricentropyminimaxrisk}, metric entropy scaling of the form 
\[
H(\varepsilon)\sim \frac{\mathfrak c\,\varepsilon^{-\alpha}}{\alpha},
\qquad \varepsilon\to0,
\]
with $\mathfrak c,\alpha>0$, is equivalent to minimax risk scaling
\[
R_\kappa
\sim
\frac{\alpha+2}{\alpha}
\left(
\frac{\mathfrak c\,\alpha\,\kappa^2}{(\alpha+1)(\alpha+2)}
\right)^{\!\frac{2}{\alpha+2}},
\qquad \kappa\to0.
\]
In view of \eqref{eq:pinsker-entropy}, for $\ballsch$ we may choose
\[
\mathfrak c=\frac{\omega_{2d}}{(2\pi\sqrt c)^d},
\qquad
\alpha=\frac{2d}{s},
\]
which yields
\[
R_\kappa(\ballsch)
\sim
\frac{d+s}{d}
\left(
\frac{d\,s\,\omega_{2d}\,\kappa^2}
{(2\pi\sqrt c)^d(d+s)(2d+s)}
\right)^{\!\frac{s}{d+s}},
\qquad \kappa\to0.
\]
This establishes \eqref{eq:pinsker-unbounded} and completes the proof of
Theorem~\ref{thm:sob-sch}.

\subsection{Proof of Theorem~\ref{thm:sob-weighted}}\label{sec:proof-sob-weighted}

It follows from Rellich's compactness theorem (see \cite[Theorem~2.17 and Exercise~2.12]{hinzmicrolocal}) that $\ballweight$ is compact in $L^2(\R^d)$.
Proceeding as in the proof of Theorem~\ref{thm:sob-sch}, with
$\symbolnotsobw$ in place of $\symbolnotsch$,
we reduce the problem to determining the asymptotic behavior of
\begin{equation}\label{eq:integral-volume-weighted-sob1}
V_{\symbolnotsobw^{-1}}(\lambda)
=
\frac{1}{(2\pi\sqrt{c})^{d}}
\int_{\mathbb{R}^{2d}}
\mathbbm{1}_{\{(1+\|\omega'\|_2^2)^{-s/2}(1+\|x'\|_2^2)^{-r/2}>\lambda\}}
\,dx'\,d\omega',
\qquad \lambda>0.
\end{equation}
We rewrite the integral in polar coordinates to get
\[
\int_0^\infty\!\!\int_0^\infty
\mathbbm{1}_{\{(1+u_0^2)^{-s/2}(1+v_0^2)^{-r/2}>\lambda\}}
\,s_{d-1}^2 (u_0v_0)^{d-1}\,du_0\,dv_0,
\]
where $s_{d-1}$ denotes the surface measure of the unit sphere in $\mathbb{R}^d$.
With the change of variables $u=u_0^d$ and $v=v_0^d$, we get
\begin{equation}\label{eq:double-integral-sob-var-ch}
    V_{\symbolnotsobw^{-1}}(\lambda)
    =
    \frac{\omega_d^2}{(2\pi\sqrt{c})^{d}}
    \int_0^\infty\!\!\int_0^\infty
    \mathbbm{1}_{\{(1+u^{2/d})^{-s/2}(1+v^{2/d})^{-r/2}>\lambda\}}
    \,du\,dv,
\end{equation}
using $\omega_d = s_{d-1}/d$.


Let us first assume that $r<s$. For $v\ge0$, the inequality
$(1+u^{2/d})^{-s/2}(1+v^{2/d})^{-r/2}>\lambda$
has solutions in $u$ only if
$\lambda(1+v^{2/d})^{r/2} < 1 $. Solving this inequality for $v$
gives $v < (\lambda^{-2/r}-1)^{d/2}$, and we therefore define
\begin{equation}\label{eq:def-v-lambda}
    v_\lambda \coloneqq (\lambda^{-2/r}-1)^{d/2}.
\end{equation}
Fix $v\in[0,v_\lambda)$. Then the condition
$(1+u^{2/d})^{-s/2}(1+v^{2/d})^{-r/2}>\lambda$
is equivalent to
$1+u^{2/d}<\lambda^{-2/s}(1+v^{2/d})^{-r/s}$,
and hence to
\[
u < \bigl(\lambda^{-2/s}(1+v^{2/d})^{-r/s}-1\bigr)^{d/2}.
\]
Consequently, integrating first with respect to $u$ gives
\[
V_{\symbolnotsobw^{-1}}(\lambda)
=
\frac{\omega_d^2}{(2\pi\sqrt{c})^{d}}
\int_0^{v_\lambda}
\left(
\lambda^{-2/s}(1+v^{2/d})^{-r/s}-1
\right)^{d/2}
\,dv,
\qquad \lambda>0.
\]
After the change of variables $t = v/v_\lambda$, the integral becomes
\begin{equation}\label{eq:vsigma-integral}
V_{\symbolnotsobw^{-1}}(\lambda)
=
\frac{\omega_d^2\, v_\lambda}{(2\pi\sqrt{c})^{d}}
\int_0^{1} u_\lambda(t)\,dt,
\qquad \lambda>0,
\end{equation}
where
\begin{equation}\label{eq:def-u-lambda}
u_\lambda(t)
\coloneqq
\left(
\lambda^{-2/s}\left(1+(t v_\lambda)^{2/d}\right)^{-r/s}-1
\right)^{d/2},
    \qquad t\in[0,1).
\end{equation}
Note that the integral \eqref{eq:vsigma-integral} includes the endpoint $t=1$, although the admissible region corresponds to $t\in[0,1)$. This is justified since the singleton $\{1\}$ has Lebesgue measure zero and $u_\lambda$ extends continuously to $[0,1]$ with $u_\lambda(1)=0$.
From the definition of $v_\lambda$ in \eqref{eq:def-v-lambda} we have $v_\lambda \sim \lambda^{-d/r}$, as 
$\lambda \to 0$. Letting $\gamma \coloneqq 2r/(ds)$ and using \eqref{eq:def-u-lambda}, we obtain the
pointwise limit
\[
u_\lambda(t) \longrightarrow
\left(t^{-\gamma}-1\right)^{d/2},
\qquad \lambda\to0,
\quad t\in(0,1).
\]
Since $r<s$, we have $\gamma<2/d$, and therefore the function
$t\mapsto (t^{-\gamma}-1)^{d/2}$ is integrable on $(0,1)$.
Moreover,
\[
\left(1+(t v_\lambda)^{2/d}\right)^{-r/s}
\le
t^{-\gamma}\left(1+v_\lambda^{2/d}\right)^{-r/s}
=
t^{-\gamma}\lambda^{2/s},
\qquad \lambda>0,\; t\in(0,1),
\]
which implies
\[
u_\lambda(t)
\le
\left(t^{-\gamma}-1\right)^{d/2}.
\]
Since $u_\lambda(t)\to (t^{-\gamma}-1)^{d/2}$ pointwise on $(0,1)$ as
$\lambda\to0$, and $|u_\lambda(t)|\le (t^{-\gamma}-1)^{d/2}$ with the latter
integrable on $(0,1)$, the hypotheses of the dominated convergence theorem are
satisfied. Consequently,
\[
\int_0^{1} u_\lambda(t)\,dt
\;\longrightarrow\;
\int_0^{1} (t^{-\gamma}-1)^{d/2}\,dt,
\qquad \lambda\to0 .
\]
Using the representation of $V_{\symbolnotsobw^{-1}}(\lambda)$ derived above
and applying the dominated convergence theorem, we obtain
\begin{equation}\label{eq:first-asymp-vol-r-leq-s}
V_{\symbolnotsobw^{-1}}(\lambda)
=
\frac{\omega_d^2\, v_\lambda}{(2\pi\sqrt{c})^{d}}
\int_0^{1} u_\lambda(t)\,dt
\;\sim\;
\frac{\omega_d^2\, v_\lambda}{(2\pi\sqrt{c})^{d}}
\int_0^{1} (t^{-\gamma}-1)^{d/2} dt,
\qquad \lambda \to 0 .
\end{equation}
Changing variables according to $t_0=t^\gamma$ in the integral on the
right-hand side of \eqref{eq:first-asymp-vol-r-leq-s}, we obtain
\begin{equation}\label{eq:binet-1}
    \int_0^{1} (t^{-\gamma}-1)^{d/2} dt
    =
    \frac{1}{\gamma}
    \int_0^{1}
    t_0^{\frac{1}{\gamma}-1-\frac{d}{2}}
    (1-t_0)^{d/2}
    \,dt_0
    =
    \frac{1}{\gamma}
    B \left(\frac{1}{\gamma}-\frac{d}{2}, \frac{d}{2}+1\right),
\end{equation}
where $B$ denotes Euler’s Beta function. Using the relation between the Beta
and Gamma functions (see \cite[Chapter~2]{artin}), 
we further get 
\begin{equation}\label{eq:binet-2}
    \frac{1}{\gamma}
    B \left(\frac{1}{\gamma}-\frac{d}{2}, \frac{d}{2}+1\right)
    =
    \frac{
        \Gamma \left(\frac{1}{\gamma}-\frac{d}{2}\right)
        \Gamma \left(\frac{d}{2}+1\right)
    }{
        \gamma\,\Gamma \left(\frac{1}{\gamma}+1\right)
    }
    =
    \frac{
        d\,\Gamma \left(\frac{d(s-r)}{2r}\right)\Gamma \left(\frac{d}{2}\right)
    }{
        2\,\Gamma \left(\frac{ds}{2r}\right)
    }.
\end{equation}
Combining \eqref{eq:def-v-lambda}, \eqref{eq:first-asymp-vol-r-leq-s},
\eqref{eq:binet-1}, and \eqref{eq:binet-2}, we conclude that
\begin{equation}\label{eq:final-asymp-vol-r-leq-s}
    V_{\symbolnotsobw^{-1}}(\lambda)
    \sim
    \frac{
        d\,\omega_d^2\,
        \Gamma \left(\frac{d(s-r)}{2r}\right)
        \Gamma \left(\frac{d}{2}\right)
    }{
        2(2\pi\sqrt{c})^{d}\,
        \Gamma \left(\frac{ds}{2r}\right)
    }
    \lambda^{-d/r},
    \qquad \lambda \to 0,
\end{equation}
in the case $r<s$.

By symmetry, the case $r>s$ is treated in the same way after interchanging the
roles of $r$ and $s$ throughout the preceding argument. In particular, the
auxiliary quantities $v_\lambda$, $u_\lambda$, and $\gamma$ are redefined with
$r$ and $s$ swapped, so that $\gamma=2s/(dr)$. This yields
\begin{equation}\label{eq:final-asymp-vol-r-geq-s}
    V_{\symbolnotsobw^{-1}}(\lambda)
    \sim \frac{d \,\omega_d^2\,  \Gamma\left(\frac{d(r-s)}{2s}\right) \Gamma \left( \frac{d}{2} \right)}{2(2\pi\sqrt{c})^{d}\, \Gamma \left(\frac{dr}{2s}\right)}\lambda^{-d/s},
    \quad \lambda \to 0.
\end{equation}

For $r=s$, we first observe that
\[
(1+u^{2/d})^{-s/2}\sim(1+u)^{-s/d},
\qquad u\to\infty .
\]
Hence, for every fixed $\eta\in(0,1)$ there exists $R_\eta>0$ such that
\begin{equation}\label{eq:r-eta-inequalities}
(1-\eta)(1+u^{2/d})^{-s/2}
\le
(1+u)^{-s/d}
\le
(1+\eta)(1+u^{2/d})^{-s/2},
\qquad u\ge R_\eta ,
\end{equation}
and the same estimates hold with $u$ replaced by $v$. Now, whenever
$u,v\ge R_\eta$, \eqref{eq:r-eta-inequalities} and its
version with $u$ replaced by $v$ imply that
\begin{align*}
    (1-\eta)^2 (1+u^{2/d})^{-s/2}(1+v^{2/d})^{-s/2}
&\le
(1+u)^{-s/d}(1+v)^{-s/d} \\
&\le
(1+\eta)^2 (1+u^{2/d})^{-s/2}(1+v^{2/d})^{-s/2}.
\end{align*}
Consequently, for $u,v\ge R_\eta$, we obtain the set inclusions
\begin{equation}\label{eq:set-inclusion-1}
\{(1+u^{2/d})^{-s/2}(1+v^{2/d})^{-s/2}>\lambda\}
\subseteq
\{(1+u)^{-s/d}(1+v)^{-s/d}>(1-\eta)^2\lambda\}
\end{equation}
and
\begin{equation}\label{eq:set-inclusion-2}
\{(1+u)^{-s/d}(1+v)^{-s/d}>(1+\eta)^2\lambda\}
\subseteq
\{(1+u^{2/d})^{-s/2}(1+v^{2/d})^{-s/2}>\lambda\}.
\end{equation}
We next decompose the double integral \eqref{eq:double-integral-sob-var-ch} defining
$V_{\symbolnotsobw^{-1}}(\lambda)$ according to
\[
V_{\symbolnotsobw^{-1}}(\lambda)
=
\frac{\omega_d^2}{(2\pi\sqrt{c})^{d}}
\left(
\int_{0}^{R_\eta}\!\!\int_{0}^{R_\eta}
+
2\int_{R_\eta}^{\infty}\!\!\int_{0}^{R_\eta}
+
\int_{R_\eta}^{\infty}\!\!\int_{R_\eta}^{\infty}
\right)
\mathbbm{1}_{\{(1+u^{2/d})^{-s/2}(1+v^{2/d})^{-s/2}>\lambda\}}
\,du\,dv .
\]
The first integral satisfies
\begin{equation}\label{eq:first-integral}
\int_{0}^{R_\eta}\!\!\int_{0}^{R_\eta}
\mathbbm{1}_{\{(1+u^{2/d})^{-s/2}(1+v^{2/d})^{-s/2}>\lambda\}}
\,du\,dv
\longrightarrow
R_\eta^2,
\qquad \lambda\to0 .
\end{equation}
Next, for the mixed region we have
\begin{align}\label{eq:second-integral}
\int_{R_\eta}^{\infty}\!\int_{0}^{R_\eta}
\mathbbm{1}_{\{(1+u^{2/d})^{-s/2}(1+v^{2/d})^{-s/2}>\lambda\}}
\,du\,dv \nonumber
&\le
R_\eta
\int_{R_\eta}^{\infty}
\mathbbm{1}_{\{(1+v^{2/d})^{-s/2}>\lambda\}}\,dv \\
&=
O_{\lambda\to0} \left(\lambda^{-d/s}\right).
\end{align}
It remains to analyze the tail region
\[
\int_{R_\eta}^{\infty}\!\!\int_{R_\eta}^{\infty}
\mathbbm{1}_{\{(1+u^{2/d})^{-s/2}(1+v^{2/d})^{-s/2}>\lambda\}}
\,du\,dv .
\]
To this end, we introduce the tail part of the corresponding 
comparison integral
\[
\bar I_\eta(\lambda)
:=
\int_{R_\eta}^{\infty}\!\!\int_{R_\eta}^{\infty}
\mathbbm{1}_{\{(1+u)^{-s/d}(1+v)^{-s/d}>\lambda\}}
\,du\,dv .
\]
For $\lambda$ sufficiently small, the admissible region in this tail integral is given by
\[
R_\eta \le v < \frac{\lambda^{-d/s}}{1+R_\eta}-1,
\qquad
R_\eta \le u < \frac{\lambda^{-d/s}}{1+v}-1,
\]
and therefore
\[
\bar I_\eta(\lambda)
=
\int_{R_\eta}^{\frac{\lambda^{-d/s}}{1+R_\eta}-1}
\left(\frac{\lambda^{-d/s}}{1+v}-1-R_\eta\right)\,dv
=
\frac{d}{s}\lambda^{-d/s}\ln(\lambda^{-1})
+ O_{\lambda\to0} \left(\lambda^{-d/s}\right),
\]
so that
\begin{equation}\label{eq:tail-comparison-integral}
\bar I_\eta(\lambda)
\sim
\frac{d}{s}\lambda^{-d/s}\ln(\lambda^{-1}),
\qquad \lambda\to0 .
\end{equation}
We now use the set inclusions \eqref{eq:set-inclusion-1} and
\eqref{eq:set-inclusion-2} to get
\[
\bar I_\eta \bigl((1+\eta)^2\lambda\bigr)
\;\le\;
\int_{R_\eta}^{\infty}\!\int_{R_\eta}^{\infty}
\mathbbm{1}_{\{(1+u^{2/d})^{-s/2}(1+v^{2/d})^{-s/2}>\lambda\}}
\,du\,dv
\]
and
\[
\int_{R_\eta}^{\infty}\!\int_{R_\eta}^{\infty}
\mathbbm{1}_{\{(1+u^{2/d})^{-s/2}(1+v^{2/d})^{-s/2}>\lambda\}}
\,du\,dv
\;\le\;
\bar I_\eta \bigl((1-\eta)^2\lambda\bigr).
\]
By \eqref{eq:tail-comparison-integral}, this implies
\[
(1+\eta)^{-2d/s}
\le
\liminf_{\lambda\to0}
\frac{
\displaystyle
\int_{R_\eta}^{\infty}\!\!\int_{R_\eta}^{\infty}
\mathbbm{1}_{\{(1+u^{2/d})^{-s/2}(1+v^{2/d})^{-s/2}>\lambda\}}
\,du\,dv
}{
\frac{d}{s}\lambda^{-d/s}\ln(\lambda^{-1})
}
\]
and
\[
\limsup_{\lambda\to0}
\frac{
\displaystyle
\int_{R_\eta}^{\infty}\!\!\int_{R_\eta}^{\infty}
\mathbbm{1}_{\{(1+u^{2/d})^{-s/2}(1+v^{2/d})^{-s/2}>\lambda\}}
\,du\,dv
}{
\frac{d}{s}\lambda^{-d/s}\ln(\lambda^{-1})
}
\le
(1-\eta)^{-2d/s}.
\]
Since $\eta\in(0,1)$ is arbitrary, letting $\eta\downarrow0$ yields
\begin{equation}\label{eq:limit-double-int-simple-form-2}
\int_{R_\eta}^{\infty}\!\!\int_{R_\eta}^{\infty}
\mathbbm{1}_{\{(1+u^{2/d})^{-s/2}(1+v^{2/d})^{-s/2}>\lambda\}}
\,du\,dv
\sim
\frac{d}{s}\lambda^{-d/s}\ln(\lambda^{-1}),
\qquad \lambda\to0.
\end{equation}
Finally, combining \eqref{eq:first-integral}, \eqref{eq:second-integral}, and \eqref{eq:limit-double-int-simple-form-2}, we can conclude that
\begin{equation}\label{eq:asymp-vol-case-r-equals-s}
V_{\symbolnotsobw^{-1}}(\lambda)
\sim
\frac{d\,\omega_d^2}{s(2\pi\sqrt{c})^{d}}
\lambda^{-d/s}\ln \left(\lambda^{-1}\right),
\qquad \lambda\to0.
\end{equation}
We now apply Theorem~\ref{thm:main-result} with
$\symbolnot=\symbolnotsobw^{-1}$ to obtain
\begin{equation}\label{eq:appli-thm1-sob2}
    H_{\symbolnotsobw^{-1}}(\varepsilon)
    \sim
    \int_\varepsilon^\infty
    \frac{V_{\symbolnotsobw^{-1}}(\lambda)}{\lambda}\,d\lambda,
    \qquad \varepsilon\to0.
\end{equation}
Combining \eqref{eq:final-asymp-vol-r-leq-s},
\eqref{eq:final-asymp-vol-r-geq-s},
\eqref{eq:asymp-vol-case-r-equals-s},
and \eqref{eq:appli-thm1-sob2}, we arrive at
\[
    H_{\symbolnotsobw^{-1}}(\varepsilon)
    \sim
    \frac{\omega_d^2}{(2\pi\sqrt{c})^{d}}
    \begin{dcases}
        \Xi_{r,s,d}\,\varepsilon^{-d/\min\{r,s\}},
        & \text{if } r\neq s,\\[.25cm]
        \varepsilon^{-d/s}\ln(\varepsilon^{-1}),
        & \text{if } r=s,
    \end{dcases}
    \qquad \varepsilon\to0,
\]
where
\[
    \Xi_{r,s,d}
    \coloneqq
    \frac{
        \min\{r,s\}\,
        \Gamma \left(\frac{d|s-r|}{2\min\{r,s\}}\right)
        \Gamma \left(\frac{d}{2}\right)
    }{
        2\,
        \Gamma \left(\frac{d\max\{r,s\}}{2\min\{r,s\}}\right)
    }.
\]
Finally,
\[
H \left(\varepsilon;\mathcal{B}_{\mathcal{W}}^{(s,r)},\|\cdot\|_{L^2(\mathbb R^d)}\right)
\sim
H_{\symbolnotsobw^{-1}}(\varepsilon),
\qquad \varepsilon \to 0,
\]
which completes the proof.
\color{black}

\section*{Acknowledgments}

The authors thank A.~Künzi for useful suggestions.

\bibliography{main}

\appendix

\section{Review of Spectral Theory for Pseudodifferential Operators}
\label{sec:review-lit}

This appendix reviews the spectral-theoretic material required for the
development of our results.
We begin by recalling the Kohn--Nirenberg representation of a
pseudodifferential operator.
With the Fourier transform defined by
\begin{equation}\label{eq:definition-fourier-transform}
    \hat f(\omega)
    =
    \int_{\R^d} f(y)\, e^{-2\pi i\, y\cdot \omega}\,\diff y,
    \qquad \omega\in\R^d,
\end{equation}
the operator $T_\sigma$ associated with a symbol $\sigma$ admits the integral
representation
\begin{equation}\label{eq:kohn-nirenberg-def2}
    (T_\sigma f)(x)
    =
    \int_{\R^{2d}} \sigma(x,\omega)\,
    e^{2\pi i (x-y)\cdot\omega}\,
    f(y)\,\diff y\,\diff\omega,
    \qquad x\in\R^d.
\end{equation}

Following standard practice (see, e.g.,
\cite[Definition~23.1]{shubin_pseudodifferential_2001}),
we assume throughout that the symbols under consideration satisfy appropriate
regularity and growth conditions.
Specifically, for $N\in\N^*$ (typically $N=2d$ or $N=3d$ below), $m\in\R$, and
$\rho\in(0,1]$, we define the symbol class
$\psidosymbclass(\R^N)$ as the set of real-valued functions
$\sigma\in C^\infty(\R^N)$ such that, for every multi-index
$\alpha\in\N^N$, there exists a constant $C_\alpha>0$ with
\begin{equation}\label{eq:decay-cond-symbol-classs}
    \bigl|\partial^\alpha \sigma(z)\bigr|
    \le
    C_\alpha\,\bigl(1+\|z\|_2^2\bigr)^{\frac{m-\rho|\alpha|}{2}},
    \qquad z\in\R^N.
\end{equation}

For later use, it is convenient to work with a slightly more general integral
representation than \eqref{eq:kohn-nirenberg-def2}, allowing the symbol to depend
on the integration variable $y$ as well.
Specifically, we consider operators of the form
\begin{equation}\label{eq:amplitude-integral-representation}
    (T_{\bar\sigma} f)(x)
    =
    \int_{\R^{2d}} \bar\sigma(x,y,\omega)\,
    e^{2\pi i (x-y)\cdot\omega}\,
    f(y)\,\diff y\,\diff\omega,
\end{equation}
where $\bar\sigma\in\psidosymbclass(\R^{3d})$ is referred to as the
amplitude of the operator. The precise definition of an amplitude varies somewhat across the literature;
see, for example, \cite[Definition~23.3]{shubin_pseudodifferential_2001} and
\cite[Definition~2.1]{treves2013psido}.
The choice adopted here—namely, requiring $\bar\sigma$ to belong to
$\psidosymbclass(\R^{3d})$—is slightly more restrictive than
\cite[Definition~23.3]{shubin_pseudodifferential_2001}, but this distinction plays
no role in the arguments below. Since $\bar\sigma\in\psidosymbclass(\R^{3d})$, it satisfies 
\eqref{eq:decay-cond-symbol-classs} with $N=3d$.
Operators of the form \eqref{eq:amplitude-integral-representation} with amplitudes $\bar\sigma\in\psidosymbclass(\R^{3d})$ define the
class $\psidoopclass(\R^d)$.
Finally, we set
\[
\Psi^{-\infty}
\;\coloneqq\;
\bigcap_{m\in\R}\psidoopclass(\R^d),
\]
which can be shown to be independent of the parameter $\rho$; see the discussion following
\cite[Definition~23.4]{shubin_pseudodifferential_2001}.

The amplitude formulation \eqref{eq:amplitude-integral-representation} provides a
unified framework encompassing several standard quantization schemes.
The Kohn--Nirenberg representation \eqref{eq:kohn-nirenberg-def2} is recovered when
the amplitude is independent of the variable \(y\), that is, when there exists a
symbol \(\sigma_L\in\Gamma_\rho^m(\R^{2d})\) such that
\[
\bar\sigma(x,y,\omega)=\sigma_L(x,\omega).
\]
In this case, \(\sigma_L\) is referred to as the Kohn--Nirenberg symbol (or
left symbol) of the operator \(T_{\bar\sigma}\).
If \(\sigma_W\in\Gamma_\rho^m(\R^{2d})\) and
\[
\bar\sigma(x,y,\omega)=\sigma_W \left(\tfrac{x+y}{2},\omega\right),
\]
then \(T_{\bar\sigma}\) coincides with the Weyl quantization of the symbol
\(\sigma_W\), which is therefore called the Weyl symbol of the operator.
Similarly, when the amplitude depends only on the second spatial variable,
\[
\bar\sigma(x,y,\omega)=\sigma_R(y,\omega),
\qquad \sigma_R\in\Gamma_\rho^m(\R^{2d}),
\]
the function \(\sigma_R\) is referred to as the right symbol of
\(T_{\bar\sigma}\).
Conversely, every operator \(T_{\bar\sigma}\in\Psi_\rho^m(\R^d)\) admits equivalent
representations in terms of a left symbol, a Weyl symbol, and a right symbol; see
\cite[Theorem~23.1]{shubin_pseudodifferential_2001}.

A central theme in the spectral theory of pseudodifferential operators is that the
asymptotic behavior of the eigenvalue-counting function can be characterized in
terms of the associated symbol, independently of the chosen quantization
(Kohn--Nirenberg, Weyl, or right).
We highlight two results of this type.
The first, stated as Theorem~\ref{thm-shubin}, is a classical Weyl-type result for
hypoelliptic pseudodifferential operators and is included here for context.
The second, Theorem~\ref{thm:Dauge-Robert}, concerns compact pseudodifferential
operators and provides the spectral characterization used in the proof of our
main result.

To state Theorem~\ref{thm-shubin}, we first introduce the notion of a
hypoelliptic symbol; see, for example,
\cite[Definition~25.1]{shubin_pseudodifferential_2001}.
A real-valued function $\sigma\in C^\infty(\R^{2d})$ is said to be hypoelliptic if
there are constants $C_1,C_2,\rho,R>0$ and real numbers $m_+\ge m_-$ such that,
for every multi-index $\alpha\in\N^{2d}$, one can find a constant $C_\alpha>0$
for which
\begin{equation}\label{eq:bound-def-hypoellip}
    C_1\bigl(1+\|z\|_2^2\bigr)^{\frac{m_-}{2}}
    \le
    |\sigma(z)|
    \le
    C_2\bigl(1+\|z\|_2^2\bigr)^{\frac{m_+}{2}},
    \quad\text{and}\quad
    \frac{|\partial^\alpha\sigma(z)|}{|\sigma(z)|}
    \le
    C_\alpha\bigl(1+\|z\|_2^2\bigr)^{-\frac{\rho|\alpha|}{2}},
\end{equation}
for all $z\in\R^{2d}$ with $\|z\|_2\ge R$.
We denote this class of symbols by $\hyposymbclass(\R^{2d})$.
By direct comparison of \eqref{eq:bound-def-hypoellip} with the estimate
\eqref{eq:decay-cond-symbol-classs}, every hypoelliptic symbol is a symbol in the
sense of $\Gamma_\rho^{m_+}(\R^{2d})$, that is,
$\hyposymbclass(\R^{2d})\subset\Gamma_\rho^{m_+}(\R^{2d})$.
A hypoelliptic operator is a pseudodifferential operator whose symbol is hypoelliptic.

Depending on the asymptotic behavior of the symbol at infinity,
hypoelliptic symbols exhibit qualitatively different spectral properties.
We therefore distinguish between symbols that grow at infinity and symbols
that decay at infinity.

\medskip

\noindent
\textbf{Positive-order hypoelliptic symbols.}
A hypoelliptic symbol $\sigma$ is said to be of positive order if it belongs
to the class
\[
\hypoposclass(\R^{2d})
\;\coloneqq\;
\bigcup_{\rho\in(0,1]}
\ \bigcup_{\substack{m_+,m_->0\\ m_+\,\ge\, m_-}}
\mathrm{H}\Gamma^{m_-,m_+}_\rho(\R^{2d}).
\]
That is, $\sigma$ is of positive order if there exist parameters $\rho\in(0,1]$
and real numbers $m_+\ge m_->0$ such that
\[
\sigma \in \mathrm{H}\Gamma^{m_-,m_+}_\rho(\R^{2d}).
\]
For symbols in this class, we have
\[
|\sigma(z)| \to \infty,
\qquad \text{as } \|z\|_2 \to \infty.
\]


\noindent 
\textbf{Negative-order hypoelliptic symbols.}
A hypoelliptic symbol $\sigma$ is said to be of negative order if it belongs
to the class
\begin{equation}\label{eq:negative-order-def-class}
\hyponegclass(\R^{2d})
\;\coloneqq\;
\bigcup_{\rho\in(0,1]}
\ \bigcup_{\substack{m_+,m_->0\\ m_+\,\ge\, m_-}}
\mathrm{H}\Gamma^{-m_+,-m_-}_\rho(\R^{2d}).
\end{equation}
That is, $\sigma$ is of negative order if there exist parameters $\rho\in(0,1]$
and real numbers $m_+\ge m_->0$ such that
\[
\sigma \in \mathrm{H}\Gamma^{-m_+,-m_-}_\rho(\R^{2d}).
\]
For symbols in this class, we have
\[
|\sigma(z)| \to 0,
\qquad \text{as } \|z\|_2 \to \infty.
\]



The classical spectral theory of hypoelliptic pseudodifferential operators
predominantly addresses the 
positive-order case. 
This stands in contrast to the negative-order setting of the present paper, where the symbols of
interest decay at infinity, so that the volume integral in
\eqref{eq:definition-volume-function} is finite.
Nevertheless, in order to place our results in context and to recall the
classical Weyl-type asymptotics available in the 
positive-order case, it is
convenient to introduce counterparts of the eigenvalue-counting function
\eqref{eq:definition-ecf} and of the volume function
\eqref{eq:definition-volume-function}, which we denote by a superscript
$\uparrow$. 
Specifically, for a Weyl symbol $\sigma_W\in\hypoposclass(\R^{2d})$, we define the
upward eigenvalue-counting function of the pseudodifferential operator
$T_{\sigma_W}$ by
\[
    \ecfup_{\sigma_W}(\lambda)
    \;\coloneqq\;
    \#\bigl\{n\in\N^* : \lambda_n(T_{\sigma_W}) < \lambda\bigr\},
    \qquad \lambda>0,
\]
where $\{\lambda_n(T_{\sigma_W})\}_{n\in\N^*}$
denotes the eigenvalues
of $T_{\sigma_W}$.
It follows from \cite[Corollary~23.3 and Theorem~26.3]{shubin_pseudodifferential_2001}
that $T_{\sigma_W}$ has purely discrete spectrum.
The upward volume associated with the same symbol $\sigma_W$ is defined by
\begin{equation}\label{eq:definition-upward-volume}
    \volup_{\sigma_W}(\lambda)
    \;\coloneqq\;
    \int_{\R^{2d}} \mathbbm{1}_{\{\sigma_W(x,\omega) \, < \, \lambda\}}
    \,\diff x\,\diff\omega,
    \qquad \lambda>0.
\end{equation}

With these notions in place, a classical Weyl-type result in the spectral theory
of hypoelliptic pseudodifferential operators can be stated as follows.


\begin{theorem}[{\cite[Theorem~30.1]{shubin_pseudodifferential_2001}}]
\label{thm-shubin}
Let $\sigma\in\hypoposclass(\R^{2d})$.
Assume that there exist constants $R>0$ and $C>0$ such that
\begin{equation}\label{eq:assumption-shubin}
    |z\cdot\nabla\sigma(z)| \ge C\,|\sigma(z)|,
    \qquad \text{for all } z\in\R^{2d}\text{ with }\|z\|_2\ge R.
\end{equation}
Then the upward eigenvalue-counting function $\ecfup_\sigma$ of the
corresponding Weyl-quantized operator satisfies
\[
    \ecfup_\sigma(\lambda)\sim \volup_\sigma(\lambda),
    \qquad \text{as } \lambda\to\infty.
\]
\end{theorem}

\noindent
Note that our definition of the upward volume in
\eqref{eq:definition-upward-volume} differs from that used in
\cite[Theorem~30.1]{shubin_pseudodifferential_2001} by a multiplicative factor of
$(2\pi)^d$.
This discrepancy reflects the different normalization conventions for the
Fourier transform; see \eqref{eq:definition-fourier-transform} and
\cite[eq.~(1.2)]{shubin_pseudodifferential_2001}.

The second result we invoke is due to Dauge and Robert \cite{dauge1987weyl}.
Their work concerns compact pseudodifferential operators of negative order and
establishes a Weyl-type asymptotic formula for the eigenvalue-counting function,
including a remainder estimate. Earlier results in this direction are due to Birman and Solomyak
\cite{birman_solomyak_1980_anisotropic}, who analyzed compact pseudodifferential
operators with symbols that are homogeneous or quasi-homogeneous with respect to
the phase variable.
A typical example of such behavior is given by symbols $\sigma$ for which there
exists $\zeta<0$ such that
\[
\sigma(x,t\omega)=t^{\zeta}\sigma(x,\omega),
\]
for all $t\ge1$, $x\in\R^d$, and $\omega\in\R^d$ with $\|\omega\|_2$ sufficiently
large. Dauge and Robert relax the requirement of exact homogeneity and work under more
general decay assumptions at infinity, thereby extending Weyl-type spectral
asymptotics to a broader class of compact pseudodifferential operators. This
flexibility makes their result applicable to the negative-order symbol regime
considered in the present paper.

Dauge--Robert formulate their results in terms of general weight functions
\(\phi,\varphi\in C^\infty(\R^{2d})\) taking values in \(\R_+^*\); see
\cite[Definition~1.1]{dauge1987weyl}.
In this framework, a function \(w\in C^\infty(\R^{2d})\) with values in \(\R_+^*\)
is called \((\phi,\varphi)\)-continuous if there exist constants \(c,C>0\) such that
\[
    C^{-1}
    \le
    \frac{w(x+x',\omega+\omega')}{w(x,\omega)}
    \le
    C
\]
for all \(x,x',\omega,\omega'\in\R^d\) satisfying
\[
    \frac{\|x'\|_2}{\varphi(x,\omega)}
    +
    \frac{\|\omega'\|_2}{\phi(x,\omega)}
    \le c .
\]
The weight \(w\) is said to be \((\phi,\varphi)\)-temperate if there exist constants
\(C,\zeta>0\) so that
\begin{equation}\label{eq:def-phi-phi-temperate}
    w(x+x',\omega+\omega')
    \le
    C\, w(x,\omega)
    \bigl[1+\|x'\|_2\,\phi(x,\omega)
          +\|\omega'\|_2\,\varphi(x,\omega)\bigr]^{\zeta},
\end{equation}
for all \(x,x',\omega,\omega'\in\R^d\).

Further, given a \((\phi,\varphi)\)-temperate weight \(w\), Dauge--Robert define the symbol
class \(\Gamma(\R^{2d};w,\phi,\varphi)\) as the set of real-valued functions
\(\sigma\in C^\infty(\R^{2d})\) such that, for every pair of multi-indices
\(\alpha,\beta\in\N^{d}\), one has
\begin{equation}\label{eq:weight-class-DR}
    \bigl|\partial_x^\alpha \, \partial_\omega^\beta \, \sigma(x,\omega)\bigr|
    \;\leq C_{\alpha,\beta}
    \frac{w(x,\omega)}
         {\phi(x,\omega)^{|\alpha|}\,\varphi(x,\omega)^{|\beta|}},
    \qquad x,\omega\in\R^d,
\end{equation}
with a constant $C_{\alpha,\beta} > 0$.

In order to control the spectral asymptotics of compact pseudodifferential
operators with symbol decay, Dauge--Robert impose a set of structural conditions
on the weight functions \(w,\phi,\varphi\) governing localization and scaling in
phase space. For ease of reference, we reproduce these assumptions below, using
the same labels as in \cite{dauge1987weyl}.

\begin{citemize}
    \item[(H1)]
    The functions \(\phi^{-1}\) and \(\varphi^{-1}\) are
    \((\phi,\varphi)\)-continuous, \((1,1)\)-temperate, and bounded on
    \(\R^{2d}\).

    \item[(H2)]
    There exist constants \(C,\zeta>0\) such that
    \begin{equation}\label{eq:DR-H2}
        C\,(1+\|x\|_2+\|\omega\|_2)^{\zeta}
        \le \phi(x,\omega)\,\varphi(x,\omega),
        \qquad x,\omega\in\R^{d}.
    \end{equation}

    \item[(W)]
    The weight \(w\) belongs to the symbol class
    \(\Gamma(\R^{2d};w,\phi,\varphi)\) and is
    \((\phi,\varphi)\)-temperate.

    \item[(N)]
    There exist constants \(K,K',\gamma,\gamma'>0\) such that
    \begin{equation}\label{eq:DR-N}
        K'\,w(x,\omega)^{\gamma'}
        \le \phi^{-1}(x,\omega)\,\varphi^{-1}(x,\omega)
        \le K\,w(x,\omega)^{\gamma},
        \qquad x,\omega\in\R^{d}.
    \end{equation}
\end{citemize}

Under assumptions \emph{(H1)}–\emph{(N)}, Dauge–Robert establish the following Weyl-type
asymptotic formula for the eigenvalue-counting function of compact
pseudodifferential operators whose symbol
$\sigma$ belongs to $\Gamma(\R^{2d}; w, \phi, \varphi)$.


\begin{theorem}[{\cite[Theorem~1.3]{dauge1987weyl}}]\label{thm:Dauge-Robert}
    Let $w, \phi, \varphi \in C^\infty(\R^{2d})$
    be weight functions satisfying (H1), (H2), (W), and (N).
    Let $\sigma\in\Gamma(\R^{2d}; w, \phi, \varphi)$ 
    be non-negative
    and assume that the maps $\lambda \mapsto V_\sigma(\lambda)$ 
    and $\lambda \mapsto V_w(\lambda)$ are regularly varying at zero with $V_w (\lambda) = O_{\lambda\to 0} \left(V_\sigma (\lambda)\right)$.
    Then, $M_{|T_\sigma|} (\lambda) \sim M_{T_\sigma} (\lambda) \sim V_{\sigma} (\lambda)$, as $\lambda\to 0$, where $T_\sigma$ is the Weyl quantization of $\sigma$.
\end{theorem}


\noindent
We remark that both Theorems~\ref{thm-shubin} and~\ref{thm:Dauge-Robert} are stated
here in slightly simplified forms relative to their original counterparts in
\cite[Theorem~30.1]{shubin_pseudodifferential_2001} and
\cite[Theorem~1.3]{dauge1987weyl}, respectively.
More precisely, \cite[Theorem~30.1]{shubin_pseudodifferential_2001} replaces the
condition~\eqref{eq:assumption-shubin} by the weaker requirement
\begin{equation}\label{eq:assumption-shubin2}
    |z \cdot \nabla \sigma(z)|
    \ge
    C\,|\sigma(z)|^{1-\delta},
    \qquad
    \text{for all } z \in \R^{2d} \text{ with } \|z\|_2 \ge R,
\end{equation}
for some $\delta \in [0,\rho')$, where $\rho'>0$ is a parameter depending on the
symbol $\sigma$ (with the universal lower bound $\rho' \ge \rho/m_+$).
Under this assumption, Shubin establishes the refined asymptotic expansion
\[
    \ecfup_\sigma(\lambda)
    =
    \volup_\sigma(\lambda)
    \left(
        1
        +
        O_{\lambda\to\infty}
        \bigl(\lambda^{\delta-\rho'+\eta}\bigr)
    \right),
    \qquad
    \text{for every } \eta>0.
\]
Similarly, \cite[Theorem~1.3]{dauge1987weyl} provides a remainder estimate of the
form
\[
    M_\sigma(\lambda)
    =
    V_\sigma(\lambda)
    +
    O_{\lambda\to0}
    \left(
        \lambda^{\zeta}\,V_w(\lambda)
    \right),
    \qquad
    \text{for some } \zeta>0,
\]
under their original set of assumptions. 
In the present paper, only first-order asymptotics are required, and the finer
remainder estimates available in the original results are not needed.
For this reason, the simplified formulations given in
Theorems~\ref{thm-shubin} and~\ref{thm:Dauge-Robert} are sufficient for our
purposes. 
Additionally, we note that the regular-variation assumption imposed in
Theorem~\ref{thm:Dauge-Robert} is stronger than the original condition~(T) in
\cite{dauge1987weyl}. This strengthening is introduced here in order to enable the
subsequent connection between eigenvalue-counting functions and metric entropy,
and therefore does not entail any loss of generality in the context of our
results. Finally, the non-negativity assumption on the symbol $\sigma$ can in principle be
removed. In \cite{dauge1987weyl}, this is achieved by separating the asymptotic
analysis of the positive and negative parts of the spectrum of $T_\sigma$. 
When
$\sigma$ is strictly positive and of negative order, the negative spectrum 
does not contribute to the leading asymptotics,
and $M_{|T_\sigma|} (\lambda) \sim M_{T_\sigma} (\lambda)$. Since our analysis
focuses exclusively on the leading-order behavior, the formulation of
Theorem~\ref{thm:Dauge-Robert} for non-negative symbols is compatible with the
assumptions of Theorem~\ref{thm:main-result}.

\end{document}